%% file: arxiv-chr-comb.tex
\documentclass[a4paper]{article}
\usepackage{fullpage}
\usepackage{authblk}
\usepackage{graphicx}
\usepackage{color}
\usepackage{amsmath,amsthm,amssymb}
\usepackage{mathtools}
\usepackage[hidelinks]{hyperref}
\usepackage{xspace}
\usepackage{epsfig}
\usepackage[dvipsnames]{xcolor}

\newcommand {\mm}[1]   {\ifmmode{#1}\else{\mbox{\(#1\)}}\fi}

\newcommand {\ceiling}[1] {{\left\lceil #1 \right\rceil}}
\newcommand {\floor}[1] {{\left\lfloor #1 \right\rfloor}}

\newcommand{\ignore}[1]{}

\newcommand{\Rspace}        {\mm{{\mathbb R}}}

\newcommand{\domain}[2]     {\mm{{\rm dom}{({#1},{#2})}}}
\newcommand{\Voronoi}[2]    {\mm{{\rm Vor}_{#1}{({#2})}}}
\newcommand{\Overlay}[3]    {\mm{{\rm Vor}_{#1}{({#2}\mid{#3})}}}
\newcommand{\Delaunay}[2]   {\mm{{\rm Del}_{#1}{({#2})}}}
\newcommand{\Membrane}[1]   {\mm{M{({#1})}}}

\newcommand{\Crossings}     {\mm{\varrho_{\rm cross}}}

\newcommand{\Exp}[1]        {\mm{{\mathbb E}{[{#1}]}}}
\newcommand{\intensity}     {\mm{\varrho}}
\newcommand{\VV}[2]         {\mm{V_{#1,#2}}}
\newcommand{\DD}[2]         {\mm{D_{#1,#2}}}
\newcommand{\XX}[1]         {\mm{X_{#1}}}

\newcommand{\length}[1]     {\mm{\rm length}{[{#1}]}}
\newcommand{\Union}[2]      {\mm{\rm Union}{({#1},{#2})}}

\newcommand{\card}[1]       {\mm{{\#}{#1}}}
\newcommand{\dime}[1]       {\mm{\rm dim\,}{#1}}

\newcommand{\affine}[1]     {\mm{\rm aff\,}{#1}}
\newcommand{\Edist}[2]      {\mm{\|{#1}-{#2}\|}}

\newcommand{\Skip}[1]       {}
\newcommand{\HE}[1]         {{\textcolor{blue}{{{\rm {#1}}}}}}

\definecolor{blue-green}{rgb}{0.0, 0.87, 0.87}

\theoremstyle{definition}

\newtheorem{theorem}{Theorem}
\numberwithin{theorem}{section}

\newtheorem{lemma}[theorem]{Lemma}

\numberwithin{equation}{section}

\title{On the Size of Chromatic Delaunay Mosaics}

\author{Ranita Biswas}
\author{Sebastiano Cultrera di Montesano}
\author{Ond\v{r}ej Draganov}
\author{Herbert~Edelsbrunner}
\author{Morteza Saghafian}
\affil{ISTA (Institute of Science and Technology Austria), Kloster\-neu\-burg, Austria}

\date{}

\begin{document}
\maketitle

\begin{abstract}
Given a locally finite set $A \subseteq \Rspace^d$ and a coloring $\chi \colon A \to \{0,1,\ldots,s\}$, we introduce the \emph{chromatic Delaunay mosaic} of $\chi$, which is a Delaunay mosaic in $\Rspace^{s+d}$ that represents how points of different colors mingle.
  Our main results are bounds on the size of the chromatic Delaunay mosaic, in which we assume that $d$ and $s$ are constants.
  For example, if $A$ is finite with $n = \card{A}$, and the coloring is random, then the chromatic Delaunay mosaic has $O(n^{\ceiling{d/2}})$ cells in expectation.
  In contrast, for Delone sets and Poisson point processes in $\Rspace^d$, the expected number of cells within a closed ball is only a constant times the number of points in this ball.
  Furthermore, in $\Rspace^2$ all colorings of a dense set of $n$ points have chromatic Delaunay mosaics of size $O(n)$.
  This encourages the use of chromatic Delaunay mosaics in applications.
\end{abstract}

\medskip
{\footnotesize
    \noindent {\bf Funding.}
    This project has received funding from the European Research Council (ERC) under the European Union's Horizon 2020 research and innovation programme, grant no.\ 788183, from the Wittgenstein Prize, Austrian Science Fund (FWF), grant no.\ Z 342-N31, and from the DFG Collaborative Research Center TRR 109, `Discretization in Geometry and Dynamics', Austrian Science Fund (FWF), grant no.\ I 02979-N35.}

%% \clearpage
%% \tableofcontents
%% \clearpage

%%%%%%%%%%%%%%%%%%%%%%%%%%%%%%%%%%%%%%%%%%%%%%%
%%%%%%%%%%%%%%%%%%%%%%%%%%%%%%%%%%%%%%%%%%%%%%%
\section{Introduction}
\label{sec:1}
%%%%%%%%%%%%%%%%%%%%%%%%%%%%%%%%%%%%%%%%%%%%%%%
%%%%%%%%%%%%%%%%%%%%%%%%%%%%%%%%%%%%%%%%%%%%%%%

The work described in this paper is motivated by applications in biology and medicine, and by structural obstacles encountered in related topological constructions.
The motivating applications have to do with the interaction of the members of a small number of different populations, such as cell types that segregate during early development \cite{Hei12}, or the tumor immune cell microenvironment in cancer \cite{Bin18}.
The challenge the interaction poses to topological data analysis has to do with the maps between the various sets, which either do not exist or require high-dimensional complexes whose sheer size is prohibitive for applications.

\medskip
A solution for two colors was proposed by Reani and Bobrowski \cite{ReBo21}, which we generalize to arbitrarily many colors and whose structural and combinatorial properties we study.
Given a locally finite set in $\Rspace^d$ and a coloring with $s+1$ colors, this generalization places the points of different colors on $s+1$ parallel copies of $\Rspace^d$, which intersect an orthogonal copy of $\Rspace^s$ at the vertices of the standard $s$-simplex.
This is a locally finite set in $\Rspace^{s+d}$, and the \emph{chromatic Delaunay mosaic} of the colored set in $\Rspace^d$ is, by definition, the Delaunay mosaic of the set in $\Rspace^{s+d}$.
A similar set-up was used in \cite{CEF01} for the purpose of geometric morphing between $s+1$ shapes, so our work also sheds light on that proposal to construct a shape space.
The structural results we wish to highlight are as follows:
\medskip \begin{itemize}
  \item the chromatic Delaunay mosaic contains the chromatic Delaunay mosaic as well as the Delaunay mosaic of any subset of the $s+1$ colors as a subcomplex; in particular, it contains the Delaunay~mosaic of each color individually and of all colors as subcomplexes;
  \item the $d$-dimensional section of the colorful cells in the chromatic Delaunay mosaic (the cells that have at least one vertex of each color) is dual to the overlay of the $s+1$ mono-chromatic Voronoi tessellations.
\end{itemize} \medskip
Our combinatorial results help gauge the extent to which chromatic Delaunay mosaics can be used in applications.
By the \emph{size} of a mosaic we mean the number or density of cells, which we relate to the number or density of the points.
Further important parameters are $d$ and $s$, which we assume are contants:
\medskip \begin{itemize}
  \item if the coloring of $n$ points in $\Rspace^d$ is random, then the chromatic Delaunay mosaic has expected size $O(n^{\ceiling{d/2}})$;
  \item for a Delone set in $\Rspace^d$, the expected number of cells of the chromatic Delaunay mosaic whose circumcenters project into a ball containing $n$ points of the Delone set is $O(n)$;
  \item for a dense set of $n$ points in $\Rspace^2$, the latter result can be strengthened to the worst-case size of the chromatic Delaunay mosaic being $O(n)$;
  \item a stationary Poisson point process in $\Rspace^d$ with bounded intensity has a chromatic Delaunay mosaic whose cells have bounded density, and we give explicit expressions for any number of colors in two dimensions and for two colors in any dimension.
\end{itemize} \medskip
To illustrate the results on Poisson point processes, we present computational experiments with bi- and tri-colored Poisson point processes in $\Rspace^2$ and $\Rspace^3$.
Note the conspicuous absence of the number of colors in all bounds given above, and this despite the fact that the chromatic Delaunay mosaic is an $(s+d)$-dimensional complex.

\medskip \noindent \textbf{Outline.}
Section~\ref{sec:2} presents general background on Delaunay mosaics and Voronoi tessellations.
Section~\ref{sec:3} introduces the chromatic Delaunay mosaic and proves some of its structural properties.
Section~\ref{sec:4} proves combinatorial bounds for the size of chromatic Delaunay mosaics.
Section~\ref{sec:5} studies the size of chromatic Delaunay mosaics for Poisson point processes and presents related computational experiments.
Section~\ref{sec:6} concludes the paper.

%% \clearpage
%%%%%%%%%%%%%%%%%%%%%%%%%%%%%%%%%%%%%%%%%%%%%%%
%%%%%%%%%%%%%%%%%%%%%%%%%%%%%%%%%%%%%%%%%%%%%%%
\section{Background}
\label{sec:2}
%%%%%%%%%%%%%%%%%%%%%%%%%%%%%%%%%%%%%%%%%%%%%%%
%%%%%%%%%%%%%%%%%%%%%%%%%%%%%%%%%%%%%%%%%%%%%%%

We need basic facts about Voronoi tessellations and their dual Delaunay mosaics in Euclidean space, and refer to \cite{AKL13} for further reading on the subject.

%%%%%%%%%%%%%%%%%%%%%%%%%%%%%%%%%%%%%%%%%%%%%%%
\subsection{Voronoi Tessellations}
\label{sec:2.1}
%%%%%%%%%%%%%%%%%%%%%%%%%%%%%%%%%%%%%%%%%%%%%%%

Letting $A \subseteq \Rspace^d$ be a finite set and $b \in A$ a point, the \emph{Voronoi domain} of $b$, denoted $\domain{b}{A}$, is the~set of points, $x \in \Rspace^d$, for which $\Edist{x}{b} \leq \Edist{x}{a}$ for all $a \in A$.
Since $A$ is finite, $\domain{b}{A}$ is the intersection of finitely many closed half-spaces and thus a convex polyhedron.
This polyhedron contains a neighborhood of $b$, which implies that it is $d$-dimensional.
A \emph{supporting hyperplane} of $\domain{b}{A}$ is a $(d-1)$-plane
whose intersection with the polyhedron is non-empty but with its interior is empty.
A \emph{face} of $\domain{b}{A}$ is the intersection with a supporting hyperplane, which is a convex polyhedron of dimension $p < d$.

The \emph{Voronoi tessellation} of $A$, denoted $\Voronoi{}{A}$, is the collection of Voronoi domains, $\domain{b}{A}$ with $b \in A$.
We refer to the domains as \emph{$d$-cells} and to their $p$-dimensional faces as \emph{$p$-cells} of $\Voronoi{}{A}$.
The $0$-cells are also called \emph{vertices} and the $1$-cells are also called \emph{edges}.
While any two $d$-cells of $\Voronoi{}{A}$ have disjoint interiors, they may intersect in shared faces.
More generally, the common intersection of one or more $d$-cells is either empty or a shared face.
For every $x \in \Rspace^d$, there is a unique cell of smallest dimension that contains $x$, and this cell contains $x$ in its interior.
It follows that the interiors of the cells of $\Voronoi{}{A}$ partition $\Rspace^d$.

\medskip
Writing $n = \card{A}$, it is clear that $\Voronoi{}{A}$ has precisely $n$ $d$-cells.
For $d = 2$, this implies that there are at most $3n$ edges and at most $2n$ vertices.
More generally for $n$ points in $\Rspace^d$, the Voronoi tessellation has $O (n^{\ceiling{d/2}})$ cells.
While this bound is tight, the number of cells depends on the relative position of the points and is much smaller for many sets, including some considered in this paper.
For example, the Voronoi tessellation of $n$ points chosen uniformly at random inside the unit cube in a constant-dimensional Euclidean space has only $O (n)$ cells in expectation; see e.g.\ \cite{Dwy91}.

%%%%%%%%%%%%%%%%%%%%%%%%%%%%%%%%%%%%%%%%%%%%%%%
\subsection{Delaunay Mosaics}
\label{sec:2.2}
%%%%%%%%%%%%%%%%%%%%%%%%%%%%%%%%%%%%%%%%%%%%%%%

The \emph{Delaunay mosaic} of $A \subseteq \Rspace^d$, denoted $\Delaunay{}{A}$, is the dual of the Voronoi tessellation of $A$.
To be specific, consider a $p$-cell of $\Voronoi{}{A}$, and observe that it is the common intersection of $m \geq d-p+1$ Voronoi domains.
Assuming this collection of domains is maximal, and writing $a_1, a_2, \ldots, a_m$ for the points in $A$ that generate them, we call the convex hull of the $a_i$ the \emph{dual Delaunay cell} of the Voronoi $p$-cell.
Its dimension is $q = d-p$.
The Delaunay mosaic of $A$ is the collection of Delaunay cells dual to cells of $\Voronoi{}{A}$.

We note that $\Delaunay{}{A}$ is a polyhedral complex; that is: it consists of closed polyhedral cells such that the boundary of each cell is the union of lower-dimensional cells in the complex.
Similarly, the collection of cells of $\Voronoi{}{A}$ is a polyhedral complex, but note that $\Voronoi{}{A}$ is, by definition, only the collection of Voronoi domains, which is not a complex.

\medskip
Call a $(d-1)$-dimensional sphere \emph{empty} of points in $A$ if no point in $A$ is enclosed by the sphere.
The points may lie on the sphere or outside the sphere, but they are not allowed to lie inside the sphere.
It is not difficult to see that the convex hull of $m$ points in $A$ is a cell in $\Delaunay{}{A}$ iff these $m$ points lie on an empty $(d-1)$-sphere, while all other points in $A$ lie strictly outside this sphere.
Indeed, the center of such an empty sphere is a point in the interior of the dual Voronoi cell, and the Voronoi domains generated by the $m$ points all share the cell.

\medskip
We say $A \subseteq \Rspace^d$ is \emph{generic}, or in \emph{general position}, if no $p+2$ points of $A$ lie on a common $(p-1)$-sphere, for $1 \leq p \leq d$.
In this case, all cells in $\Delaunay{}{A}$ are simplices, so $\Delaunay{}{A}$ is a simplicial complex in $\Rspace^d$.
Correspondingly, every $p$-cell of $\Voronoi{}{A}$ is the common intersection of exactly $d-p+1$ Voronoi domains, so the common intersection of any $d+2$ Voronoi domains is necessarily empty.
This is what we call a \emph{simple} decomposition of $\Rspace^d$.
In this case, the Delaunay mosiac is isomorphic to the \emph{nerve} of the Voronoi tessellation, which consists of all collections of domains in $\Voronoi{}{A}$ that have a non-empty common intersection.
The assumption that $A$ be generic often simplifies matters, and it can be simulated computationally \cite{EdMu90} to avoid cumbersome special cases.

%% \clearpage
%%%%%%%%%%%%%%%%%%%%%%%%%%%%%%%%%%%%%%%%%%%%%%%%
%%%%%%%%%%%%%%%%%%%%%%%%%%%%%%%%%%%%%%%%%%%%%%%%
\section{Chromatic Complexes}
\label{sec:3}
%%%%%%%%%%%%%%%%%%%%%%%%%%%%%%%%%%%%%%%%%%%%%%%%
%%%%%%%%%%%%%%%%%%%%%%%%%%%%%%%%%%%%%%%%%%%%%%%%

The main concepts in this section are the
chromatic Delaunay mosaics and Voronoi tessellations, which generalize the bi-chromatic construction in \cite{ReBo21} to more than two colors.

%%%%%%%%%%%%%%%%%%%%%%%%%%%%%%%%%%%%%%%%%%%%%%%
\subsection{Chromatic Delaunay Mosaics}
\label{sec:3.1}
%%%%%%%%%%%%%%%%%%%%%%%%%%%%%%%%%%%%%%%%%%%%%%%

Throughout this section, we let $A$ be $n$ points in $\Rspace^d$, $\sigma = \{0, 1, \ldots, s\}$ a collection of colors, $\chi \colon A \to \sigma$ a coloring, and $A_j = \chi^{-1} (j)$ the subset of points with color $j$, for $0 \leq j \leq s$.
We recall that the \emph{standard $s$-simplex} is the convex hull of the $s+1$ unit coordinate vectors in $\Rspace^{s+1}$.
To map this simplex to $s$ dimensions, we identify $\Rspace^s$ with the $s$-plane defined by $x_1 + x_2 + \ldots + x_{s+1} = 1$ in $\Rspace^{s+1}$ and parametrize it with the inherited $s+1$ \emph{barycentric coordinates}.
A subset of $t+1 \leq s+1$ unit coordinate vectors defines the standard $t$-simplex, which we map to $\Rspace^t$ by parametrizing it with the $t+1$ barycentric coordinates inherited from $\Rspace^{t+1}$.
We are ready to construct the \emph{chromatic Delaunay mosaic} of $\chi$, denoted $\Delaunay{}{\chi}$, which we do by writing $\Rspace^{s+d} = \Rspace^s \times \Rspace^d$ in three steps:
\medskip \begin{description}
  \item[{\sc Step~1:}] let $u_0, u_1, \ldots, u_s$ be the vertices of the standard $s$-simplex in $\Rspace^s$;
  \item[{\sc Step~2:}] set $A' = A_0' \sqcup A_1' \sqcup \ldots \sqcup A_s'$, in which $A_j' = u_j + A_j \subseteq u_j + \Rspace^d$, for each $0 \leq j \leq s$;
  \item[{\sc Step~3:}] construct $\Delaunay{}{\chi} = \Delaunay{}{A'}$;
\end{description} \medskip
\begin{figure}[htb]
  \centering \vspace{0.0in}
  %\resizebox{!}{2.0in}{\input{Figs/chromatic.pdf_t}}
  \includegraphics[width=0.54\textwidth]{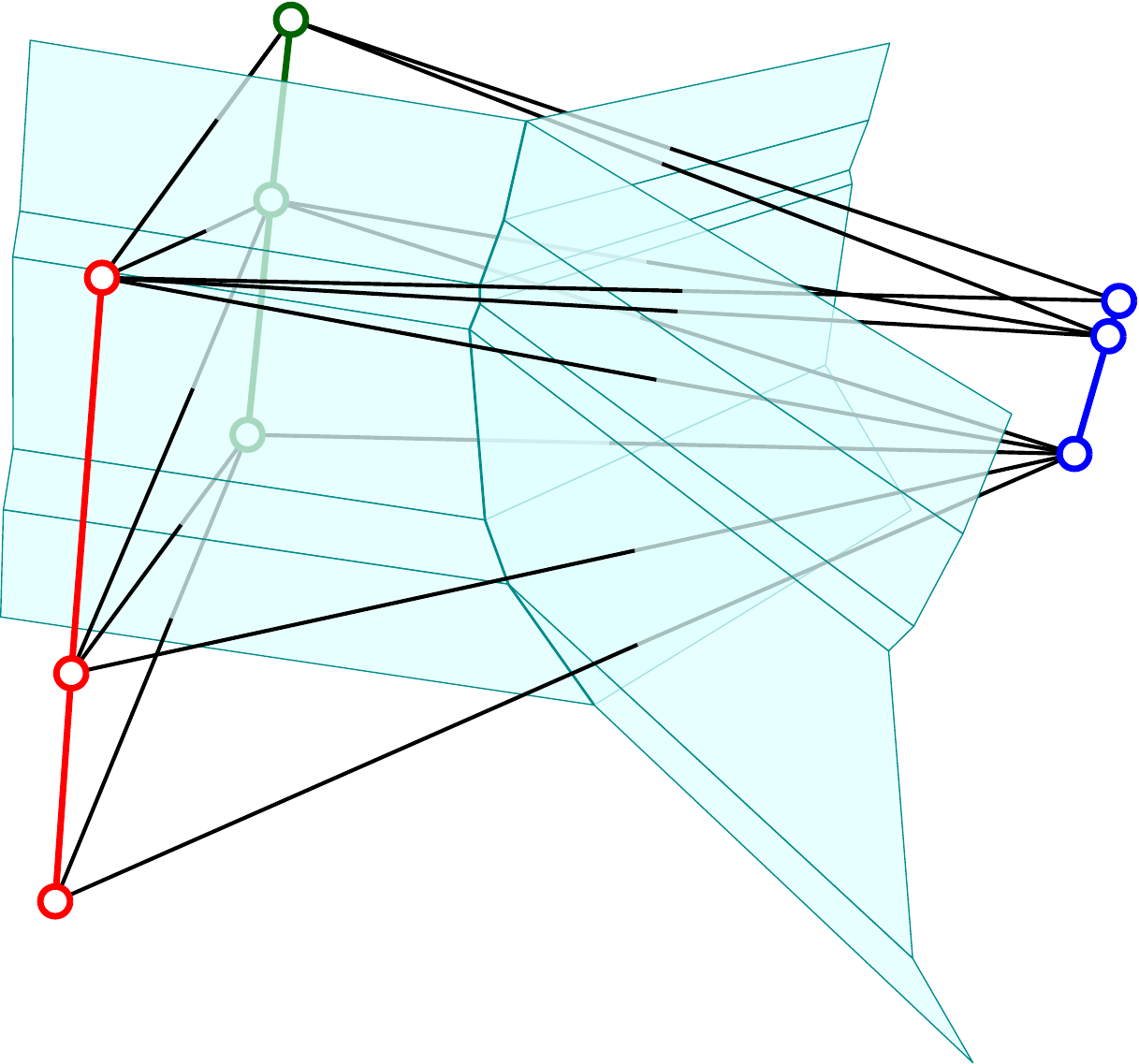}
  \caption{The chromatic Delaunay mosaic of three finite sets in $\Rspace^1$ together with the stratification of space into membranes.
  The points of each set are placed on a copy of $\Rspace^1$ orthogonal to the $2$-plane that carries the standard triangle constructed in Step~1.
  The stratification consists of a $1$-dimensional membrane geometrically located between the three lines, and three $2$-dimensional membranes, one between any two of the lines.}
  \label{fig:chromatic}
\end{figure}
see Figure~\ref{fig:chromatic}. 
Similarly, we apply the construction to a subset of the colors, $\tau \subseteq \sigma$, and write $\Delaunay{}{\chi|\tau}$, in which $\chi|\tau$ is our notation for the restriction of $\chi$ to $\chi^{-1} (\tau)$.
This mosaic lives in $\Rspace^{t+d}$, in which $t = 1+\card{\tau}$.
It is not difficult to see that $\Delaunay{}{\chi|\tau}$ is a subcomplex of $\Delaunay{}{\chi}$.
To state this property formally, we call a cell in $\Delaunay{}{\chi}$ \emph{$\tau$-colored} if the colors of its vertices belong to $\tau$, and \emph{$\tau$-colorful} if it is $\tau$-colored and has a vertex of every color in $\tau$.
Every cell is $\tau$-colorful for the smallest subset, $\tau \subseteq \sigma$, for which the cell is $\tau$-colored.
This implies that we get a partition of the cells into $2^{s+1}$ classes. 
Note that the $\tau$-colored cells form a subcomplex of $\Delaunay{}{A}$, while the $\tau$-colorful cells generally do not. 
\begin{lemma}[Sub-chromatic Delaunay Subcomplexes]
  \label{lem:sub-chromatic_Delaunay_subcomplexes}
  Let $A \subseteq \Rspace^d$ be finite, $\chi \colon A \to \sigma$ a coloring, and $\tau \subseteq \sigma$.
  Then the subcomplex of $\tau$-colored cells in $\Delaunay{}{\chi}$ is $\Delaunay{}{\chi|\tau}$.
\end{lemma}
\begin{proof}
  Let $H$ be a hyperplane in $\Rspace^{s+d}$ that passes through all points with color $\tau$ such that all other points in $A'$ are contained in an open half-space bounded by $H$.
  The cells of $\Delaunay{}{\chi|\tau}$ are characterized by the existence of an empty $(t+d-1)$-sphere in $H$ that passes through the vertices of the cell and through no other points with color in $\tau$.
  Since all points with color in $\sigma \setminus \tau$ lie in an open half-space bounded by $H$, we can extend this $(t+d-1)$-sphere to an empty $(s+d-1)$-sphere that passes through the same points.
  Hence, $\Delaunay{}{\chi|\tau} \subseteq \Delaunay{}{\chi}$, which implies the claim because $\Delaunay{}{\chi|\tau}$ exhausts all $\tau$-colored cells in $\Delaunay{}{\chi}$.
\end{proof}

It is perhaps more difficult to see how $\Delaunay{}{\chi}$ relates to $\Delaunay{}{A}$.
In the relatively straightforward simplicial case, $\Delaunay{}{\chi}$ contains a subcomplex whose projection to $\Rspace^d$ is $\Delaunay{}{A}$; see Figure~\ref{fig:delA_in_delChi}.
In the general and therefore not necessarily simplicial case, we can for example have a convex quadrangle in $\Delaunay{}{A}$ that is the projection of a tetrahedron in $\Delaunay{}{\chi}$.
We formulate the relationship that allows for this and similar cases  in terms of the nerves of $\Voronoi{}{A}$ and $\Voronoi{}{\chi}$, which are possibly high-dimensional abstract simplicial complexes.

\begin{figure}[htb]
  \centering \vspace{0.05in}
  \includegraphics[width=.3\textwidth]{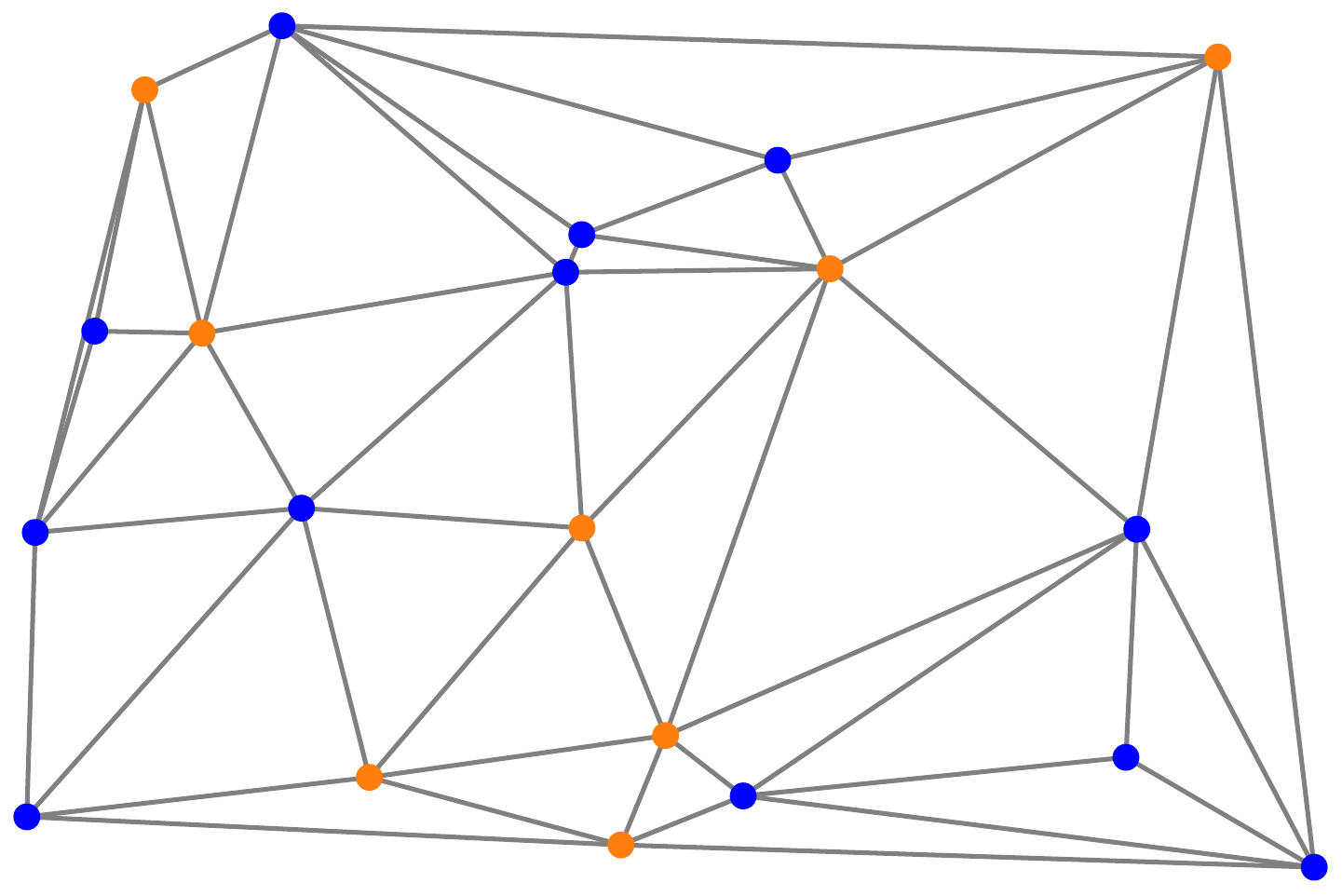}\hfill
  \includegraphics[width=.3\textwidth]{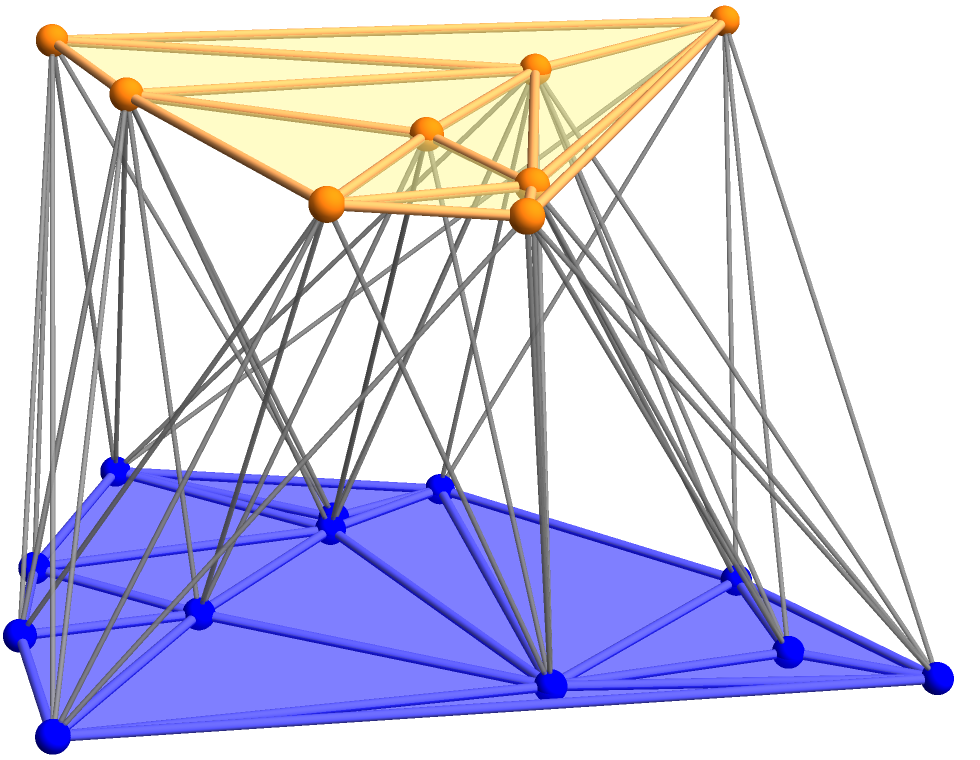}\hfill
  \includegraphics[width=.3\textwidth]{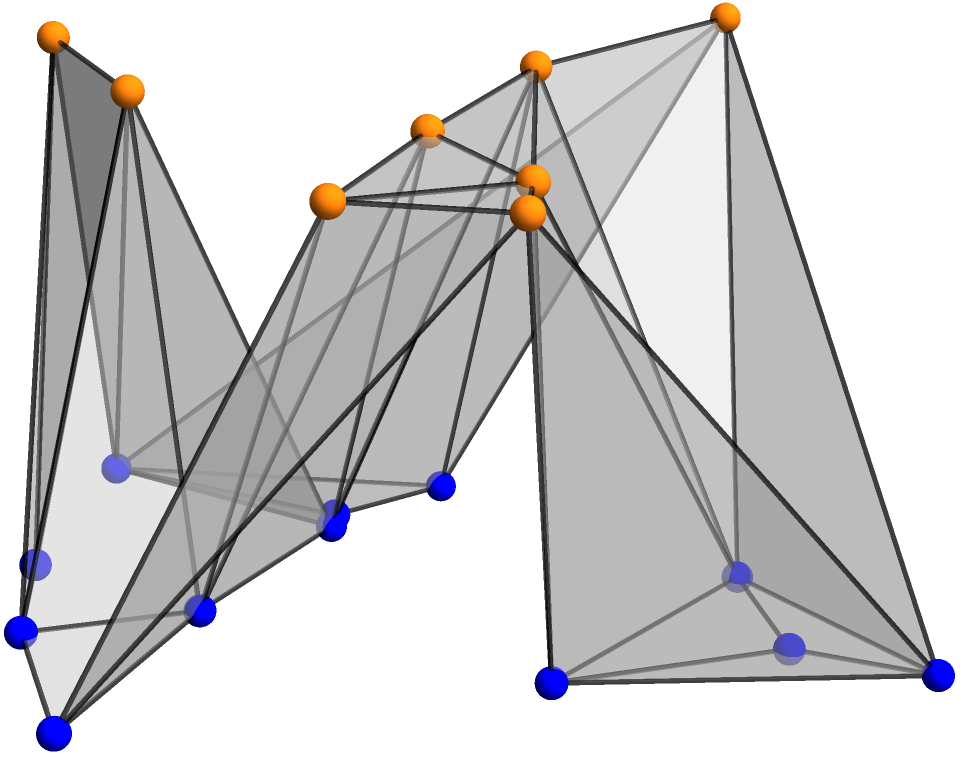}
  \vspace{-0.05in}
  \caption{\emph{Left:} the Delaunay mosaic of a bi-colored set in the plane, $\Delaunay{}{A}$.
  \emph{Middle:} the chromatic Delaunay mosaic, $\Delaunay{}{\chi}$, with colorful triangles left unfilled for clarity.
  \emph{Right:} the subcomplex of $\Delaunay{}{\chi}$ that is isomorphic to $\Delaunay{}{A}$.}
  \label{fig:delA_in_delChi}
\end{figure}

\begin{lemma}[Projection to Delaunay Mosaic]
  \label{lem:projection_to_Delaunay_mosaic}
  Let $A \subseteq \Rspace^d$ be finite, $\sigma = \{0, 1, \ldots, s\}$, and $\chi \colon A \to \sigma$ a coloring.
  Then the nerve of the $(s+d)$-cells of $\Voronoi{}{\chi}$ in $\Rspace^{s+d}$ has a subcomplex that projects to the nerve of the $d$-cells of $\Voronoi{}{A}$ in $\Rspace^d$.
\end{lemma}
\begin{proof}
  Recall that $k+1$ points in $A$ are the vertices of a cell in $\Delaunay{}{A}$ iff there is an empty $(d-1)$-sphere, $S$, that passes through these $k+1$ points and through no other points of $A$.
  The nerve of the corresponding $k+1$ Voronoi $d$-cells is a $k$-simplex.
  
  Following the construction of the chromatic Delaunay mosaic, we copy $S$ to $u_j + S$ for each $j \in \sigma$.
  Let $S'$ be the $(s+d-1)$-sphere in $\Rspace^{s+d}$ whose intersection with $u_j + \Rspace^d$ is $u_j + S$, for every $j \in \sigma$.
  It should be clear that $S'$ exists: its center projected to $\Rspace^s$ is the barycenter of the standard $s$-simplex and projected to $\Rspace^d$ is the center of $S$.
  By construction, $S'$ is empty and passes through the points $u_j + a$ with $a \in S$ and $\chi(a) = j$, and through no other points of $A'$.
  The nerve of the corresponding $(s+d)$-cells in $\Voronoi{}{\chi}$ is again isomorphic to a $k$-simplex, and its projection to $\Rspace^d$ is the $k$-simplex isomorphic to the nerve of the $k+1$ Voronoi $d$-cells we started with.
  The claim follows.
\end{proof}

%%%%%%%%%%%%%%%%%%%%%%%%%%%%%%%%%%%%%%%%%%%%%%%
\subsection{Chromatic Voronoi Tessellations}
\label{sec:3.2}
%%%%%%%%%%%%%%%%%%%%%%%%%%%%%%%%%%%%%%%%%%%%%%%

The \emph{chromatic Voronoi tessellation} of $\chi \colon A \to \sigma$ is the Voronoi tessellation of $A' \subseteq \Rspace^{s+d}$, and we write $\Voronoi{}{\chi} = \Voronoi{}{A'}$.
There is a bijection between the cells of $\Voronoi{}{\chi}$ and $\Delaunay{}{\chi}$, denoted by mapping $\nu$ to $\nu^* \in \Delaunay{}{\chi}$, such that $\dime{\nu} + \dime{\nu^*} = s+d$ and $\mu$ is a face of $\nu$ iff $\nu^*$ is a face of $\mu^*$.
We call $\nu$ \emph{$\tau$-colored} or \emph{$\tau$-colorful} if $\nu^*$ is $\tau$-colored or $\tau$-colorful, respectively.
For each $\tau \subseteq \sigma$, we defined the \emph{$\tau$-membrane} of $\chi$ as the union of the interiors of the $\tau$-colorful cells of $\Voronoi{}{\chi}$, denoted $\Membrane{\tau}$.
Since the interiors of the cells in $\Voronoi{}{\chi}$ partition $\Rspace^{s+d}$, and the interior of each cell belongs to exactly one membrane, the membranes are pairwise disjoint and partition $\Rspace^{s+d}$; see Figure~\ref{fig:chromatic}.

\begin{lemma}[Stratification into Membranes]
  \label{lem:stratification_into_membranes}
  Let $A \subseteq \Rspace^d$ be finite, $\sigma = \{0,1,\ldots,s\}$, and $\chi \colon A \to \sigma$ a coloring.
  \begin{enumerate}
    \item For each non-empty $\tau \subseteq \sigma$, $\Membrane{\tau}$ is a manifold homeomorphic to $\Rspace^{s-t+d}$, with $t = \card{\tau} - 1$.
    \item The collection of $\Membrane{\tau}$ form a stratification of $\Rspace^{s+d}$ with strata of dimension $d$ to $s+d$, in which the $p$-stratum is the disjoint union of all $\Membrane{\tau}$ with $\card{\tau} = s+d-p+1$.
  \end{enumerate}
\end{lemma}
\begin{proof}
  We begin with $\tau = \sigma$.
  Let $w \in \Rspace^d$ and consider $w + \Rspace^s$, which is an $s$-plane parallel to $\Rspace^s$ and therefore orthogonal to $\Rspace^d$.
  By Pythagoras' theorem, the squared distance between points $x \in w+\Rspace^s$ and $y \in \Rspace^d$ is $\Edist{x}{w}^2 + \Edist{w}{y}^2$.
  Letting $a$ be the point in $A$ closest to $x$, this implies that $a$ is the closest point in $A$ to any point in $w + \Rspace^d$.
  Similarly, if $a_j'$ is the point in $A_j'$ closest to $x$, then $a_j'$ is the closest point in $A_j'$ to any point in $w+\Rspace^s$.
  There is a unique point $z(w) \in w+\Rspace^s$ at equal distance to $a_0', a_1', \ldots, a_s'$.
  Hence, $z(w) \in \Membrane{\sigma}$ and it is indeed the only point of $w+\Rspace^s$ in $\Membrane{\sigma}$.
  It follows that $\Membrane{\sigma}$ is the image of $z \colon \Rspace^d \to \Rspace^{s+d}$ defined by mapping $w$ to $z(w)$.
  Note that $z$ is continuous, so $\Membrane{\sigma}$ is homeomorphic to $\Rspace^d$.
  It is the stratum of the lowest dimension, $d$, in the claimed stratification.
  
  To describe the remainder of the stratification, let $V(\sigma)$ be the Voronoi tessellation of $u_0, u_1, \ldots, u_s$ in $\Rspace^s$.
  Since the $u_j$ are the vertices of the standard $s$-simplex, this tessellation consists of a vertex at $0 \in \Rspace^s$, $s+1$ half-lines emanating from $0$, $\binom{s+1}{2}$ $2$-dimensional wedges connecting the half-lines in pairs, etc.
  Returning to $w+\Rspace^s$, we observe that it slices the stratification of $\Rspace^{s+d}$ in a translate of this $s$-dimensional Voronoi tessellation, $z(w) + V(\sigma)$.
  Varying $w$ over all points of $\Rspace^d$, we get the claimed stratification of $\Rspace^{s+d}$.
\end{proof}

%%%%%%%%%%%%%%%%%%%%%%%%%%%%%%%%%%%%%%%%%%%%%%%
\subsection{Overlay of Mono-chromatic Voronoi Tessellations}
\label{sec:3.3}
%%%%%%%%%%%%%%%%%%%%%%%%%%%%%%%%%%%%%%%%%%%%%%%

Related to the strata are the overlays of tessellations.
Given $A \subseteq \Rspace^d$, $\sigma = \{0,1,\ldots,s\}$, and $\chi \colon A \to \{0,1,\ldots,s\}$, the \emph{overlay} of the $s+1$ mono-chromatic Voronoi tessellations, denoted $\Overlay{}{A_j}{j \in \sigma}$, is the decomposition of $\Rspace^d$ obtained by drawing the Voronoi cells of dimension at most $d-1$ on top of each other; see Figure~\ref{fig:overlay}.
More formally, each $d$-dimensional cell in the overlay is the common intersection of $s+1$ $d$-cells, one in each $\Voronoi{}{A_j}$ for $j \in \sigma$, and the overlay consists of these $d$-dimensional cells and their faces.
Even if the points in $A$ are in general position, the overlay is not necessarily a simple decomposition of $\Rspace^d$.
\begin{figure}[htb]
  \centering \vspace{0.1in}
  \includegraphics[width=.6\textwidth]{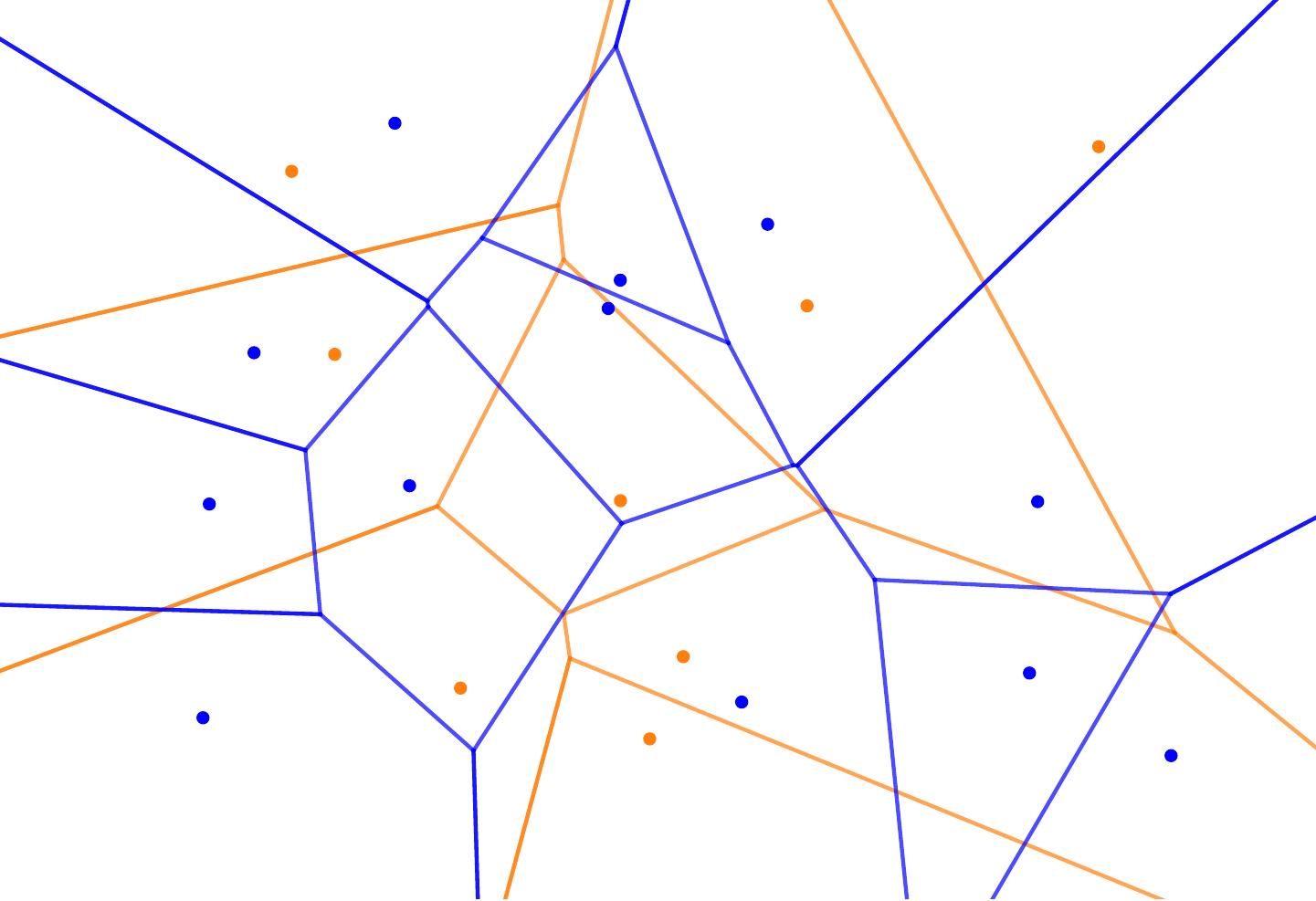}
  \caption{The overlay of a blue and a orange Voronoi tessellation in the plane.
  In the generic case, each of its vertices is either a vertex of a mono-chromatic tessellation, which has degree $3$, or the crossing of two edges, which has degree $4$.}
  \label{fig:overlay}
\end{figure}
\begin{lemma}[Membranes and Overlays]
  \label{lem:membranes_and_overlays}
  Let $A \subseteq \Rspace^d$ be finite, $\chi \colon A \to \{0,1,\ldots,s\}$ a coloring, and $A_j = \chi^{-1} (j)$ for $0 \leq j \leq s$.
  For each $\tau \subseteq \sigma$, $\Overlay{}{A_j}{j \in \tau}$ is the projection of the $\tau$-membrane, $\Membrane{\tau}$, to $\Rspace^d$.
\end{lemma}
\begin{proof}
  We begin with $\tau = \sigma$.
  By Lemma~\ref{lem:stratification_into_membranes}, $\Membrane{\sigma}$ is a manifold of dimension $d$, and in the proof of this lemma we learn that the orthogonal projection, $\pi \colon \Membrane{\sigma} \to \Rspace^d$, is a homeomorphism.
  Indeed, $\pi^{-1}$ is the restriction of $z \colon \Rspace^d \to \Rspace^{s+d}$ defined there.
  Since $\Membrane{\sigma}$ is decomposed into cells of $\Voronoi{}{\chi}$, $z$ is piecewise linear, so it suffices to prove that the linear pieces are the images of the cells in $\Overlay{}{A_j}{j \in \sigma}$.
  
  Let $\nu_j$ be a $d$-cell of $\Voronoi{}{A_j}$ and write $a_j \in A_j$ for the point that generates $\nu_j$, for $0 \leq j \leq s$.
  Assume that $\nu = \nu_0 \cap \nu_1 \cap \ldots \cap \nu_s$ has non-empty interior, so it is a $d$-cell of the overlay.
  Correspondingly, the image of every point $x \in \nu$, $z(x) = \pi^{-1}(x)$, is equidistant from the points $u_j + a_j$, for $0 \leq j \leq s$.
  It follows that the image of $\nu$ is a subset of a linear piece in $\Membrane{\sigma}$.
  For every neighboring $d$-cell of $\nu$ in the overlay, we change one of the $a_j$, so their images belong to different linear pieces of $\Membrane{\sigma}$.
  This implies that the image of $\nu$ \emph{is} a linear piece of $\Membrane{\sigma}$, as required.

  \medskip
  To generalize, let $\tau \subseteq \sigma$ and use the above argument  to conclude that $\Overlay{}{A_j}{j \in \tau}$ is the projection of the $\tau$-membrane to $\Rspace^d$.
  Recall that $\Voronoi{}{\chi | \tau}$ decomposes $\Rspace^{t+d}$, and by Lemma~\ref{lem:sub-chromatic_Delaunay_subcomplexes}, the extrusion of the $\tau$-membrane in $\Voronoi{}{\chi | \tau}$ along the remaining $s-t$ coordinate directions in $\Rspace^{s+d}$ contains the $\tau$-membrane in $\Voronoi{}{\chi}$.
  Moreover, the projections of the two $\tau$-membranes---one in $\Voronoi{}{\chi | \tau}$ and the other in $\Voronoi{}{\chi}$---to $\Rspace^d$ are identical, which implies the claim.
\end{proof}

%% \clearpage
%%%%%%%%%%%%%%%%%%%%%%%%%%%%%%%%%%%%%%%%%%%%%%%%
%%%%%%%%%%%%%%%%%%%%%%%%%%%%%%%%%%%%%%%%%%%%%%%%
\section{Counting Cells}
\label{sec:4}
%%%%%%%%%%%%%%%%%%%%%%%%%%%%%%%%%%%%%%%%%%%%%%%%
%%%%%%%%%%%%%%%%%%%%%%%%%%%%%%%%%%%%%%%%%%%%%%%%

In this section, we are interested in the size of the chromatic Delaunay mosaic or, equivalently, 
of the overlays between the mono-chromatic Voronoi tessellations.
We focus on extremal questions, in which we minimize or maximize over locally finite sets and their colorings, but we also consider random colorings.

%%%%%%%%%%%%%%%%%%%%%%%%%%%%%%%%%%%%%%%%%%%%%%%
\subsection{Few Spherical {\it k}-sets Imply Small Expected Overlays}
\label{sec:4.1}
%%%%%%%%%%%%%%%%%%%%%%%%%%%%%%%%%%%%%%%%%%%%%%%

Let $A$ be a set of $n$ points in $\Rspace^d$.
We call a subset of $k \leq n$ points a \emph{spherical $k$-set} of $A$ if there is a sphere that separates the $k$ from the remaining $n-k$ points.
Note that this differs from the classic notion of a $k$-set, for which there is a hyperplane that separates the $k$ points of the $k$-set from the remaining $n-k$ points.
In this section, we relate the number of spherical $k$-sets with the expected size of the overlay of mono-chromatic Voronoi tessellations for random colorings of $A$.
Specifically, we prove the following lemma.
\begin{lemma}[Spherical $k$-sets and Overlay]
  \label{lem:spherical_k-set_and_overlay}
  Let $c, d, e$ be positive constants, and $A$ a set of $n$ points in $\Rspace^d$ such that for every $1 \leq k \leq n$, the number of spherical $k$-sets is $O(k^c n^e)$.
  Let furthermore $s \geq 0$ be a constant, let $\sigma = \{0,1,\ldots,s\}$, and write $A_j = \chi^{-1} (j)$, in which $\chi \colon A \to \sigma$ is a random coloring.
  Then the expected size of $\Overlay{}{A_j}{j \in \sigma}$ is $O(n^e)$.
\end{lemma}
\begin{proof}
  We assume that the points in $A$ are in general position and write $A_j = \chi^{-1} (j)$.
  Suppose we pick $s+1$ cells, one from each $\Voronoi{}{A_j}$, and write $i_j - 1$ for their co-dimensions.
  The common intersection of the $s+1$ cells is either empty or a cell of co-dimension $\sum_{j=0}^s(i_j - 1)$.
  This is a vertex only if $\sum_{j=0}^s (i_j-1) = d$ or, equivalently, $\sum_{j=0}^s i_j = d+s+1$.
  To bound the expected size of the overlay, we bound the expected number of such vertices in the overlay, which is the sum of their probabilities to belong to the overlay.
  
  \medskip
  Fix $d+s+1$ points and a coloring $\chi \colon A \to \{0, 1, \ldots, s\}$ such that every color is assigned to at least one of the $d+s+1$ points.
  Writing $i_j$ for the number of points with color $j$, we have $\sum_{j=0}^s i_j = d+s+1$ and $1 \leq i_j \leq d+1$ for each $j$.
  Let $E_j$ be the set of points $y \in \Rspace^d$ at equal distance to the $i_j$ points with color $j$; it is a plane of co-dimension $i_j - 1$.
  Since $\sum_{j=0}^s (i_j - 1) = d$ and the $d+s+1$ points are in general position, the common intersection of the $E_j$ is a point $x \in \Rspace^d$.
  This point is a vertex of the overlay iff there is a stack of spheres, $S_0, S_1, \ldots, S_s$, with common center, $x$, such that $S_j$ passes through the $i_j$ points with color $j$, and all other points in $A_j = \chi^{-1} (j)$ lie outside $S_j$.
  Suppose that $S_0$ is the largest of the $s+1$ spheres.
  Let $k$ be the number of points on or inside $S_0$, note that this is a spherical $k$-set, and recall that there are at most $O(k^c n^e)$ such sets by assumption.
  Other than the $i_0 \leq d+1$ points on $S_0$, all points in the spherical $k$-set must have color different from $0$.
  The probability of this is $s/(s+1)$ to the power $k-i_0 \geq k-d-1$.
  The number of possible overlay vertices whose largest sphere of the corresponding stack of spheres separates the same spherical $k$-set is at most $\binom{k}{d+s+1} (s+1)^{d+s+1}$. This is the product of the number of subsets of size $d+s+1$ and the number of different colorings of such a set.
  Writing $X$ for the number of vertices in the overlay, we thus get
  \begin{align}
    \Exp{X} &< \sum\nolimits_{k=d+s+1}^n O(k^c n^e) \cdot \tbinom{k}{d+s+1} (s+1)^{d+s+1} \cdot \left( \frac{s}{s+1} \right)^{k-d-1} \\
    &< O(n^e) \cdot \sum\nolimits_{k=0}^\infty \frac{(s+1)^{2d+s+2}}{s^{d+1}} \cdot k^{c+d+s+1} \cdot \left( \frac{s}{s+1} \right)^{k} .
  \end{align}
  The first factor within the latter sum is constant, the second is a constant degree polynomial, and the last factor is an exponential that vanishes as $k$ goes to infinity.
  Because of the exponential decay, the sum converges to a constant that depends on $c$, $d$, and $s$ but not on $n$.
  It follows that the number of vertices in the overlay is $O(n^e)$.
  
  \medskip
  Observe that every vertex of the overlay belongs to only a constant number of cells of dimension~$1$~to~$d$.
  Every such cell has at least one vertex, which implies that the number of cells of any dimension in the overlay is $O(n^e)$.
\end{proof}

As originally proved by Lee \cite{Lee82}, the number of spherical $k$-sets of $n$ points in $\Rspace^2$ is less than $2kn$.
The expected size of the overlay of the mono-chromatic Voronoi tessellations for a random coloring in $\Rspace^2$ is therefore $O(n)$.
To get a result for general dimensions, we note that the spherical $k$-sets in $\Rspace^d$ correspond to (linear) $k$-sets in $\Rspace^{d+1}$.
For the latter, Clarkson and Shor \cite{ClSh89} proved that the number of $\ell$-sets, for $\ell = 1, 2, \ldots, k$, is $O( k^{\ceiling{(d+1)/2}} n^{\floor{(d+1)/2}} )$.
Lemma~\ref{lem:spherical_k-set_and_overlay} thus implies
\begin{theorem}[Overlay Size for Random Coloring]
  \label{thm:overlay_size_for_random_coloring}
  Let $d$ and $s$ be constants, let $A$ be a set of $n$ points in $\Rspace^d$, let $\sigma = \{0,1,\ldots,s\}$, and write $A_j = \chi^{-1} (j)$, in which $\chi \colon A \to \sigma$ is a random coloring.
  Then the expected number of cells in $\Overlay{}{A_j}{j \in \sigma}$ is $O( n^{\ceiling{d/2}})$.
\end{theorem}
This bound is asymptotically tight since even a single Voronoi tessellation of $n$ points in $\Rspace^d$ can have $\Omega (n^{\ceiling{d/2}})$ vertices, for example if the points are chosen on the moment curve, which is defined by $(t, t^2, \ldots, t^d)$, $t \in \Rspace$.

%%%%%%%%%%%%%%%%%%%%%%%%%%%%%%%%%%%%%%%%%%%%%%%
\subsection{Delone Sets Have Small Expected Overlays}
\label{sec:4.2}
%%%%%%%%%%%%%%%%%%%%%%%%%%%%%%%%%%%%%%%%%%%%%%%

We start by showing that dense sets without big holes have few spherical $k$-sets. To formalize this claim, we recall that $A \subseteq \Rspace^d$ is a \emph{Delone set} if there are constants $0 < r < R < \infty$ such that every open ball of radius $r$ contains at most one point, and every closed ball of radius $R$ contains at least one point of $A$.
For counting purposes, we say a spherical $k$-set, $B \subseteq A$, \emph{corresponds} to a point, $b \in A$, if there is a sphere that separates $B$ from $A \setminus B$ and $b$ is the point in $A$ closest to the center of this sphere.
\begin{lemma}[Bounded Correspondence]
  \label{lem:bounded_correspondence}
  Let $A \subseteq \Rspace^d$ be a Delone set.
  Then every point in $A$ corresponds to at most $O( k^{d+1} )$ spherical $k$-sets of $A$.
\end{lemma}
\begin{proof}
  Let $x$ be a point in $\Rspace^d$ and suppose that it lies in the interior of a $d$-cell of the order-$k$ Voronoi tessellation of $A$.
  Assuming this cell is $\domain{B}{A}$, then $B$ is the unique spherical $k$-set that is separated from $A \setminus B$ by a sphere with center $x$.
  Letting $t$ be the radius of this sphere, we have $k r^d \leq (t+r)^d$ because the sphere with center $x$ and radius $t+r$ encloses $k$ disjoint open balls of radius $r$.
  Furthermore, $(t-R)^d \leq k R^d$ because the closed balls of radius $R$ centered at the points of $B$ cover the ball with center $x$ and radius $t-R$.
  Hence
  \begin{align}
    (\sqrt[d]{k} - 1) r  &\leq  t  \leq  (\sqrt[d]{k} + 1) R .
  \end{align}
  Let $b \in A$ be a point with $x \in \domain{b}{A}$.
  Since $A$ is Delone, $\domain{b}{A}$ is covered by the ball with center $b$ and radius $R$.
  It follows that the sphere with center $b$ and radius $(\sqrt[d]{k} + 2) R$ encloses all spherical $k$-sets that correspond to $b$.
  The number of points in $A$ enclosed by this sphere satisfies
  \begin{align}
    \ell  &\leq  \left[ \left( \sqrt[d]{k} + 2 \right) R + r \right]^d / r^d ,
  \end{align}
  which is $O(k)$ because $r$ and $R$ and therefore $R/r$ are positive constants.
  For a finite set in $\Rspace^d$, the number of ways it can be split into two by a sphere is less than the $(d+1)$-st power of its cardinality.
  Hence, there are at most $O(k^{d+1})$ spherical $k$-sets that correspond to $b$.
\end{proof}

Delone sets are necessarily infinite, so we let $\Omega$ be the closed ball with radius $\omega$ centered at the origin, and count a spherical $k$-set, $B \subseteq A$, only if there is a sphere that separates $B$ from $A \setminus B$ whose center is in $\Omega$.
\begin{theorem}[Overlay Size for Delone Set]
  \label{thm:overlay_size_for_Delone_set}
  Let $d$ and $s$ be constants, let $A \subseteq \Rspace^d$ be a Delone set, let $\sigma = \{0,1,\ldots, s\}$, let $\chi \colon A \to \sigma$ be a random coloring, and let $\Omega$ be the ball of points at distance at most $\omega > R$ from the origin.
  Writing $n = \card{(A \cap \Omega)}$ and $A_j = \chi^{-1} (j)$, the expected number of cells in $\Overlay{}{A_j}{j \in \sigma}$ that have at least one vertex in $\Omega$ is $O(n)$.
\end{theorem}
\begin{proof}
  Let $0 < r < R < \infty$ be constants for which $A$ is Delone, and note that the number of points of $A$ at distance at most $\omega + R$ from the origin is $O(n)$.
  Any spherical $k$-set that has a separating sphere with center in $\Omega$ corresponds to a point in this slightly larger ball, so Lemma~\ref{lem:bounded_correspondence} implies that the number of such spherical $k$-sets is $O(k^{d+1} n)$.
  
  We count the vertices of the overlay using Lemma~\ref{lem:spherical_k-set_and_overlay} but restricted to crossings inside $\Omega$.
  We have $c = d+1$ and $e=1$, so we get an expected number of $O(n)$ vertices in $\Omega$.
  Assuming general position, every vertex belongs to only a constant number of cells, which implies the claimed bound on the number of cells with at least one vertex in $\Omega$.
\end{proof}

A vertex of the overlay corresponds to an $(s+d)$-cell in the chromatic Delaunay mosaic whose circumcenter project to the vertex.
Theorem~\ref{thm:overlay_size_for_Delone_set} thus counts the cells in the chromatic Delaunay mosaic that are faces of $(s+d)$-cells whose circumcenters project into $\Omega$.

%%%%%%%%%%%%%%%%%%%%%%%%%%%%%%%%%%%%%%%%%%%%%%%
\subsection{Dense Sets in the Plane Have Always Small Overlays}
\label{sec:4.3}
%%%%%%%%%%%%%%%%%%%%%%%%%%%%%%%%%%%%%%%%%%%%%%%

In $d=2$ dimensions, Theorem~\ref{thm:overlay_size_for_Delone_set} can be strengthened while weakening the assumptions on the points. The strong bound is presented in Theorem \ref{thm:overlay_size_for_dense_set}. The proof relies on two technical lemmas, which we prove first.

Let $Y \subseteq \Rspace^2$, $\varrho \colon Y \to \Rspace$ non-negative, and $\Union{Y}{\varrho}$ the union of closed disks with centers $x \in Y$ and radii $\varrho (x)$.
For example, $Y$ may be a line segment, a square, or the complement of a square, as illustrated in Figure~\ref{fig:unions}, and the radii may be any non-negative real numbers.
\begin{lemma}[Boundary of Union of Disks]
  \label{lem:boundary_of_union_of_disks}
  Let $S$ be a line segment of length $L$, $Q$ a square with sides of length $L$, and $\bar{Q}$ the closed complement of $Q$.
  For $Y \in \{S, Q, \bar{Q} \}$, let $\varrho_Y \colon Y \to \Rspace$ be non-negative and $R_Y = \max_{x \in Y} \varrho_Y (x)$.
  Then 
  \begin{align}
    \length{\partial \Union{S}{\varrho_S}} &< 4L+8R_S, 
      \label{eqn:lengthS} \\ 
    \length{\partial \Union{Q}{\varrho_Q}} &< 8L+8R_Q,
      \label{eqn:lengthQ} \\ 
    \length{\partial \Union{\bar{Q}}{\varrho_{\bar{Q}}}} &< 8L.
      \label{eqn:lengthbarQ}
  \end{align}
\end{lemma}

\begin{proof}
  We begin with the line segment, $S$, write $\affine{S}$ for the line that contains $S$, and assume that this line is horizontal.
  Note that $\partial \Union{S}{\varrho_S}$ is invariant under reflection across this line.
  Think of the boundary above the line as the graph of a function, with alternating minima and maxima as we go from left to right.
  We focus on the piece of the graph between a minimum and an adjacent maximum, and claim that this piece is at least as wide as it is high.
  To see this, note that the maximum is the center of a disk, and the piece lies on or above the upper half-circle in the boundary of this disk.
  If the entire piece lies in this half-circle, and the minimum is where the half-circle touches the line, then the width is equal to the height.
  In all other cases, the width exceeds the height.
  The length of the piece is less than its width plus its height, which is at most twice the width.
  The sum of widths is at most $L+2R_S$, which implies that the length of $\partial \Union{S}{\varrho_S}$ above $\affine{S}$ is less than $2L + 4R_S$.
  We get the same bound for the length below $\affine{S}$, which implies \eqref{eqn:lengthS}.
  \begin{figure}[htb]
    \centering \vspace{0.0in}
    \resizebox{!}{1.6in}{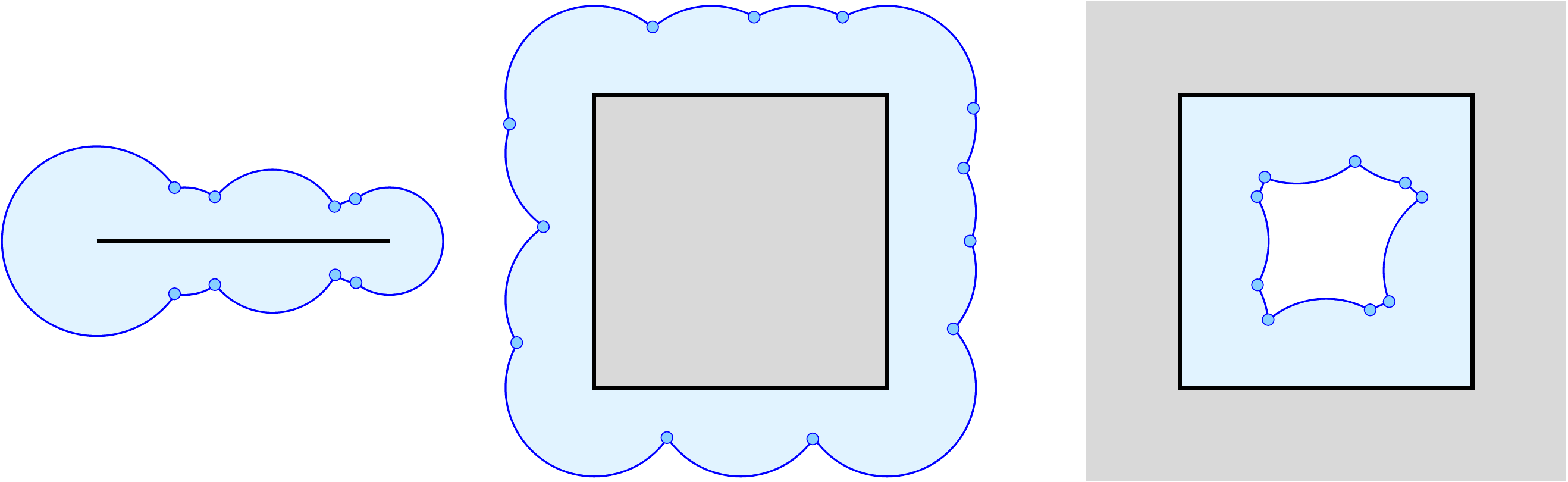}
    \caption{Unions of closed disks whose centers lie on a line segment, on the \emph{left}, in a square, in the \emph{middle}, and in the complement of a square, on the \emph{right}.
    The blue points mark the shared endpoints of the circular arcs that make up the boundary of the union of disks.}
    \label{fig:unions}
  \end{figure}
  
  To get \eqref{eqn:lengthQ}, we decompose $\partial \Union{Q}{\varrho_Q}$ into four curves by cutting along the lines that support the upper and lower sides of the square.
  By the above argument, the length of the upper curve is less than $2L+4R_Q$, and similar for the lower curve.
  The left curve has (vertical) width $L$, so we get $2L$ as an upper bound for the length, and similar for the right curve.
  The sum of the four bounds is equal to the right-hand side of \eqref{eqn:lengthQ}.
  
  \medskip
  To get \eqref{eqn:lengthbarQ}, we decompose $\bar{Q}$ into four pieces by cutting along the two lines that contain the diagonals of the square.
  For each piece, we take the union of disks with centers in the piece, and finally clip the boundary to within $Q$.
  The four curves cover $\partial \Union{\bar{Q}}{\varrho_{\bar{Q}}}$, and by the above argument, each curve has length less than $2L$.
  This implies \eqref{eqn:lengthbarQ}.
\end{proof}

Suppose there is a finite set, $B \subseteq \Rspace^2$, such that $\varrho (x)$ is the distance to the closest point in $B$.
In this case, the number of points in $B$ that lie on the boundary of $\Union{Y}{\varrho}$ relates to the number of edges in $\Voronoi{}{B}$ that cross $Y$ or the boundary of $Y$.
As before, we distinguish between a line segment, a square, and the complement of the square.
\begin{lemma}[Counting Points and Crossings]
  \label{lem:counting_points_and_crossings}
  Let $S$ be a line segment, $Q$ a square, $\bar{Q}$ the closed complement of $Q$, and $B \subseteq \Rspace^2$ finite.
  For $Y \in \{ S, Q, \bar{Q} \}$, let $\varrho_Y \colon Y \to \Rspace$ be defined by $\varrho_Y (x) = \min_{a \in B} \Edist{x}{a}$, and write $B_Y = B \cap \partial \Union{Y}{\varrho_Y}$.
  Then the number of edges of $\Voronoi{}{B}$ that have a non-empty intersection with $S$, $\partial Q$, $\partial \bar{Q}$ is bounded from above by $\card{B_S}$, $\card{B_Q}$, $\card{B_{\bar{Q}}}$, respectively.
\end{lemma}

\begin{proof}
  We begin with the line segment, $S$, and as before we assume that $\affine{S}$ is horizontal.
  By construction, $\Union{S}{\varrho_S}$ is contractible and symmetric with respect to $\affine{S}$.
  It follows that $\partial \Union{S}{\varrho_S}$ is a closed curve with even number of circular arcs meeting at the same number of vertices.
  In the generic case, every point of $B$ on $\partial \Union{S}{\varrho_S}$ is a vertex of the curve, and a vertex is a point in $B$ iff the reflected vertex on the other side of $\affine{S}$ is not a point in $B$.
  If we replace a point of $B$ that is a vertex of the curve by its reflected copy, then this changes the Voronoi tessellation but not the way in which $S$ crosses its edges.
  We can therefore assume that all vertices above $\affine{S}$ are points in $B$, while all vertices below $\affine{S}$ are not.
  In this case, we have a crossing for each pair of adjacent vertices above $\affine{S}$.
  The number of crossings is thus less than the number of points of $B$ on the boundary of the union of disks, which implies the first claim.
  
  %\medskip
  Consider next the case of a square, $Q$.
  As before, we decompose $\partial \Union{Q}{\varrho_Q}$ into four curves, one above the line of the upper side, the second below the line of the lower side, and the remaining left and right curves.
  For each curve, we reflect points of $B$ so that all vertices shared between adjacent circular arcs are points in $B$.
  In this case, we have four more points of $B$ on $\partial \Union{Q}{\varrho_Q}$ than edges of $\Voronoi{}{B}$ that cross the sides of $Q$.
  The argument for the complement of $Q$ is similar and omitted.
\end{proof}

Define the \emph{density} of a finite set as the maximum distance between two points divided by the minimum distance between two points.
A set in $\Rspace^2$ is \emph{dense} if its density is not much bigger than $\sqrt{n}$.
We show that for a dense set in the plane, the overlay of mono-chromatic Voronoi tessellations has small size for \emph{every} coloring.

\begin{theorem}[Overlay size for Dense set]\label{thm:overlay_size_for_dense_set}
  Let $A \subseteq \Rspace^2$ be finite with density $m$, write $n = \card{A}$, let $\sigma = \{0,1,\ldots,s\}$, and write $A_j = \chi^{-1} (j)$, in which $\chi \colon A \to \sigma$ is any coloring.
  Then the number of regions in $\Overlay{}{A_j}{j \in \sigma}$ is $O(s^2 m^2)$.
\end{theorem}
\begin{proof}
  We will show that the number of crossings between the edges of any two mono-chromatic Voronoi tessellations is $O(m^2)$.
  There are $\binom{s+1}{2} \leq s^2$ pairs of colors and therefore $O(s^2 m^2)$ crossings in total.
  The number of regions in the overlay is the number of regions in the $s+1$ mono-chromatic Voronoi tessellations, which is $n = \card{A}$, plus twice the number of crossings.
  Since $n = O(m^2)$, this implies that the number of regions is $O(s^2 m^2)$.
  
  \medskip
  For the remainder of this proof, we fix two colors, $0$ and $1$, we assume that the minimum distance between points in $A$ is $1$, so the maximum distance is $m$.
  Observe that there is a square with sides of length $m$ that contains $A$, and therefore $A_0$ and $A_1$.
  If there is at least one point each of $A_0$ and $A_1$ in the square, then we subdivide it into four equally large squares.
  We recursively subdivide each of these squares independently until we arrive at squares that contain points of at most one of these two colors.
  By choosing the initial square judiciously, we may assume that no point of $A$ lies on the boundary of any of these squares.
  Since subdivision does not alter the total area, the sum of areas of these squares is $m^2$.
  
  Let $Q$ be a square in this subdivision, write $L$ for the length of its sides, and assume that it contains no points of $A_0$.
  Let $Q'$ be the parent square of four times the area, which, by construction, contains at least one point of $A_0$ and at least one point of $A_1$.
  Let $\varrho_Q \colon Q \to \Rspace$ be defined by $\varrho_Q (x) = \min_{a \in A_0} \Edist{x}{a}$.
  Since $Q'$ contains at least one point of $A_0$, we have $R_Q = \max_{x \in Q} \varrho_Q (x) < 2 \sqrt{2} L$.
  Recall that $\Union{Q}{\varrho_Q}$ is the union of closed disks with centers $x$ and radii $\varrho_Q (x)$.
  By Lemma~\ref{lem:boundary_of_union_of_disks}, the length of the boundary satisfies
  \begin{align}
    \length{\partial \Union{Q}{\varrho_Q}}  &<  8L + 8R_Q  < (8 + 16 \sqrt{2}) L  <  31 L .
  \end{align}
  Since any two points of $A_0$ are at least a distance $1$ apart, this implies that there are fewer than $31 L$ points of $A_0$ on the boundary of the union of disks.
  By Lemma~\ref{lem:counting_points_and_crossings}, fewer than $31 L$ edges in $\Voronoi{}{A_0}$ cross the sides of $Q$.
  Since no point of $A_0$ is inside $Q$, the edges of $\Voronoi{}{A_0}$ inside $Q$ do not form cycles, so more than half of them cross the sides of $Q$.
  It follows that fewer than $62 L$ edges of $\Voronoi{}{A_0}$ have a non-empty intersection with $Q$.
  Let $S$ be the intersection of one such edge with $Q$, which is either the entire edge or a connected piece of it.
  The length of $S$ is at most $\sqrt{2} L$.
  Let $\varrho_S \colon S \to \Rspace$ be defined by $\varrho_S (x) = \min_{a \in A_1} \Edist{x}{a}$.
  The maximum such distance satisfies $R_S = \max_{x \in S} \varrho_S (x) < 2 \sqrt{2} L$.
  By~Lemma~\ref{lem:boundary_of_union_of_disks}, the length of the boundary satisfies
  \begin{align}
    \length{\partial \Union{S}{\varrho_S}}  &<  4L+8R_S  <  (4 + 16 \sqrt{2}) L  <  27 L .
  \end{align}
  Since any two points of $A_1$ are at least a distance $1$ apart, this implies that fewer than $27 L$ points of $A_1$ lie on the boundary of the union of disks.
  By Lemma~\ref{lem:counting_points_and_crossings}, fewer than $27 L$ edges of $\Voronoi{}{A_1}$ cross $S$.
  Multiplying with the number of edges in $\Voronoi{}{A_0}$ inside $Q$, we get fewer than $62 L \cdot 27 L = 1674 L^2$ crossings.
  This is only a constant times the area of $Q$.
  Taking the sum over all squares in the subdivision, we thus get fewer than $1674 m^2$ crossings between edges of $\Voronoi{}{A_0}$ and $\Voronoi{}{A_1}$ inside the initial square.
  
  \medskip
  It remains to bound the number of crossings outside the initial square.
  Let $\bar{Q}$ be the complement of the initial square, which we recall has sides of length $m$.
  Let $\varrho_j \colon \bar{Q} \to \Rspace$ be defined by $\varrho_j (x) = \min_{a \in A_j} \Edist{x}{a}$, for $j = 0, 1$.
  By Lemma~\ref{lem:boundary_of_union_of_disks}, we have
  \begin{align}
    \length{\partial \Union{\bar{Q}}{\varrho_j}}  &<  8m .
  \end{align}
  By Lemma~\ref{lem:counting_points_and_crossings}, fewer than $8m$ edges of $\Voronoi{}{A_j}$ cross the sides of $\bar{Q}$.
  Since there are no points of $A_j$ in $\bar{Q}$,
  this implies that $\Voronoi{}{A_j}$ has fewer than $16 m$ edges with non-empty intersection with $\bar{Q}$.
  Even if every such edge of $\Voronoi{}{A_0}$ crossed every such edge of $\Voronoi{}{A_1}$, we still have fewer than $16 m \cdot 16 m = 256 m^2$ crossings outside $Q$.
  Adding the numbers of crossings inside and outside $Q$, we get fewer than $1930 m^2$ crossings altogether.
\end{proof}

For dense sets, $A$, we have $m = \Theta (\sqrt{n})$, so the size of the overlay of $s+1$ mono-chromatic Voronoi tessellations is $O(s^2 n)$, which is linear in $n$ if we assume that $s$ is a constant.

%% \newpage
%%%%%%%%%%%%%%%%%%%%%%%%%%%%%%%%%%%%%%%%%%%%%%%
%%%%%%%%%%%%%%%%%%%%%%%%%%%%%%%%%%%%%%%%%%%%%%%
\section{Poisson Point Processes}
\label{sec:5}
%%%%%%%%%%%%%%%%%%%%%%%%%%%%%%%%%%%%%%%%%%%%%%%
%%%%%%%%%%%%%%%%%%%%%%%%%%%%%%%%%%%%%%%%%%%%%%%

We use a stationary Poisson point process in $\Rspace^d$ with a random coloring as the model for random data.
Recall that the \emph{intensity} of the process is the expected number of points per unit volume in $\Rspace^d$.
To make a linguistic difference, we call the expected number of vertices of the Voronoi tessellation per unit volume the \emph{density} of the vertices, and similar for the cells of dimension one or higher.
After deriving relevant densities from prior work, we present experimental findings, which confirm some of the derived densities but also go beyond them.
We note that a stationary Poisson point process on a compact domain is a sampling according to the uniform distribution.
So modulo boundary effects, the density of the process translates to a linear bound for the uniform distribution.

%%%%%%%%%%%%%%%%%%%%%%%%%%%%%%%%%%%%%%%%%%%%%%%
\subsection{Densities, Analytically}
\label{sec:5.1}
%%%%%%%%%%%%%%%%%%%%%%%%%%%%%%%%%%%%%%%%%%%%%%%

We focus on the vertices of the overlay of Voronoi tessellations.
The local neighborhood of every such vertex has constant size, which implies that the density of $p$-cells in the overlay is at most a constant times the density of the vertices.
Besides the vertices of the mono-chromatic Voronoi tessellations, there are also \emph{crossings}, which are the $0$-dimensional common intersections of two or more cells in differently colored mono-chromatic Voronoi tessellations.
Assuming general position, the sum of the co-dimensions of these cells is necessarily equal to $d$.
For every $0 \leq p \leq d$ and every $k \geq 1$, the density of the $p$-cells in the order-$k$ Voronoi tessellation of a stationary Poisson point process, $A \subseteq \Rspace^d$, is a constant times $k^{d-1}$, and an explicit formula is given in \cite[Theorem 1.2]{EdNi19B}.
Given a random coloring, the proof of  Lemma~\ref{lem:spherical_k-set_and_overlay} thus implies that the density of crossings between the mono-chromatic Voronoi tessellations is also bounded away from infinity.
For the cases in which $d=2$ or $s+1 = 2$, we can use prior work on weighted and unweighted Delaunay mosaics \cite{EdNi19A,EdNi19B} to determine these densities precisely.
In $\Rspace^2$, crossings happen between two edges, one each of two different Voronoi tessellations.
\begin{theorem}[Density of Crossings in Plane]
  \label{thm:density_of_crossings_in_plane}
  Let $A \subseteq \Rspace^2$ be a stationary Poisson point process with intensity $\intensity > 0$, $s$ a constant, and $\chi \colon A \to \{0,1,\ldots,s\}$ a random coloring.
  Then the density of crossings between the mono-chromatic Voronoi tessellations is $\Crossings = \frac{4s}{\pi} \cdot \intensity$.
\end{theorem}
\begin{proof}
  Since the coloring is random, each $A_j = \chi^{-1} (j)$ is a stationary Poisson point process with intensity $\frac{\intensity}{s+1}$; see e.g.\ \cite[Chapter 11]{Pis14}.
  By \cite[Theorem 1.1]{EdNi19B}, this implies that the density of the length of the $1$-skeleton of $\Voronoi{}{A_j}$ is $2 \sqrt{\intensity/(s+1)}$, and as proved in \cite{EdNi19A}, the density of crossings between a line and the $1$-skeleton is $\frac{4}{\pi} \sqrt{\intensity/(s+1)}$.
  This is true for every $0 \leq j \leq s$, so we get the density of crossings between the two $1$-skeletons by multiplication, which gives $\frac{8}{\pi} \frac{\intensity}{s+1}$.
  There are $\binom{s+1}{2} = \frac{1}{2} s (s+1)$ pairs of colors, which implies $\Crossings = \frac{4s}{\pi} \cdot \intensity$.
\end{proof}

The proof of Theorem~\ref{thm:density_of_crossings_in_plane} can be modified to show that the density of crossings is maximized by balanced colorings.
Suppose for example that $s+1 = 2$ and the random coloring is biased, with probabilities $\lambda$ and $1-\lambda$ for colors $0$ and $1$, respectively.
Then the intensities of $A_0$ and $A_1$ are $\lambda \varrho$ and $(1-\lambda) \varrho$, so the density of the crossings between their Voronoi tessellations is $\frac{8}{\pi} \sqrt{\lambda (1-\lambda)}$, which is a maximum for $\lambda = \frac{1}{2}$.

\medskip
We extend Theorem~\ref{thm:density_of_crossings_in_plane} to two colors in $d$ dimensions, while leaving the case of three or more colors in three or more dimensions as an open question.
We prepare the extension by introducing three families of constants, in which we write $\omega_d$ for the $(d-1)$-dimensional volume of the unit sphere in $\Rspace^d$ and $\Gamma$ for the gamma function, which generalizes the factorial to real arguments:
\begin{align}
  \VV{p}{d}  &=  \frac{2^{d-p+1} \pi^{\frac{d-p}{2}}}{d (d-p+1)!} \cdot
    \frac{\Gamma(\frac{d^2-pd+p+1}{2})}{\Gamma(\frac{d^2-pd+p}{2})} \cdot
    \frac{\Gamma(\frac{d+2}{2})^{d-p+\frac{p}{d}}}{\Gamma(\frac{d+1}{2})^{d-p}} \cdot
    \frac{\Gamma(d-p+\frac{p}{d})}{\Gamma(\frac{p+1}{2})} , 
    \label{eqn:VV} \\
  \DD{p}{d}  &=  \frac{\omega_1 \omega_{d+1}}{\omega_{p+1} \omega_{d-p+1}}
    \frac{2^{p+1} \pi^{\frac{p}{2}}}{d (p+1)!} \cdot
    \frac{\Gamma(\frac{pd+d-p+1}{2})}{\Gamma(\frac{pd+d-p}{2})} \cdot
    \frac{\Gamma(\frac{d+2}{2})^{p+1-\frac{p}{d}}}{\Gamma(\frac{d+1}{2})^{p}} \cdot
    \frac{\Gamma(p+1-\frac{p}{d})}{\Gamma(\frac{d-p+1}{2})} , 
    \label{eqn:DD} \\
  \XX{d}     &=  \tfrac{1}{2} ( \VV{1}{d} \DD{1}{d} + \VV{2}{d} \DD{2}{d} + \ldots + \VV{d-1}{d} \DD{d-1}{d} ),
    \label{eqn:XX}
\end{align}
for $d \geq 2$ and $1 \leq p \leq d-1$.
By comparing the factors of $\VV{p}{d}$ and $\DD{p}{d}$, it is not difficult to see that $\VV{p}{d} \DD{p}{d} = \VV{d-p}{d} \DD{d-p}{d}$.
Indeed, the two sides of this equation are just different ways to count the same thing, as we will see shortly.
Table~\ref{tbl:crossingconstants} gives approximations of the constants for small values of $d$ and $p$.
\begin{table}[hbt]
    \centering
    \begin{tabular}{r||rrrrr|r}
      & \multicolumn{5}{c|}{$\VV{p}{d} \cdot \DD{p}{d}$} & \\
      & \multicolumn{1}{c}{$p=1$} & \multicolumn{1}{c}{$p=2$} & \multicolumn{1}{c}{$p=3$} & \multicolumn{1}{c}{$p=4$} & \multicolumn{1}{c|}{$p=5$} & \multicolumn{1}{c}{$\XX{d}$} \\
      \hline \hline
      $d=2$ & $2.00 \cdot 1.27$ &                   &                   &                   &                    & $1.27$ \\
      $d=3$ & $5.83 \cdot 1.46$ & $2.91 \cdot 2.92$ &                   &                   &                    & $8.49$ \\
      $d=4$ & $23.96 \cdot 1.58$ & $10.97 \cdot 3.66$ & $3.72 \cdot 10.17$ &                   &                    & $57.88$ \\
      $d=5$ & $126.74 \cdot 1.67$ & $53.22 \cdot 4.25$ & $17.00 \cdot 13.30$ & $4.45 \cdot 47.53$ &                    & $437.78$ \\
      $d=6$ & $809.75 \cdot 1.74$ & $316.00 \cdot 4.74$ & $94.90 \cdot 16.11$ & $23.68 \cdot 63.20$ & $5.12 \cdot 274.93$  & $3668.63$
    \end{tabular}
    \caption{Approximations of the constants in \eqref{eqn:VV}, \eqref{eqn:DD}, \eqref{eqn:XX} for small values of $d$ and $p$.}
    \label{tbl:crossingconstants}
\end{table}
The meaning of the constants and the corresponding sources will be revealed in the proof of the next theorem.
For two Voronoi tessellations in $\Rspace^d$, crossings happen between the $p$-cells of one and the $(d-p)$-cells of the other tessellation.
\begin{theorem}[Density of Crossings for Two Colors]
  \label{thm:density_of_crossings_for_two_colors}
  Let $A \subseteq \Rspace^d$ be a stationary Poisson point process with intensity $\intensity > 0$, and let $\chi \colon A \to \{0,1\}$ be a random bi-coloring.
  Then the density of the crossings between the mono-chromatic Voronoi tessellations is $\Crossings (\chi) = \XX{d} \cdot \intensity$.
\end{theorem}
\begin{proof}
  Because the bi-coloring is random, both $A_0 = \chi^{-1}(0)$ and $A_1 = \chi^{-1}(1)$ are stationary Poisson point processes with intensity $\frac{\intensity}{2}$ in $\Rspace^d$.
  By \cite[Theorem 1.1]{EdNi19B}, the density of the $p$-dimensional volume of the $p$-skeleton of either tessellation is $\VV{p}{d} \cdot (\frac{\intensity}{2})^{{(d-p)}/{d}}$, and by \cite{EdNi19A}, the density of the crossings between a $p$-plane and the $(d-p)$-cells of either tessellation is $\DD{p}{d} \cdot (\frac{\intensity}{2})^{p/d}$.
  Multiplying the two densities and taking the sum for $1 \leq p \leq d-1$, we get $\Crossings = 2 \XX{d} \cdot \frac{\intensity}{2} = \XX{d} \cdot \intensity$.
\end{proof}

%%%%%%%%%%%%%%%%%%%%%%%%%%%%%%%%%%%%%%%%%%%%%%%%
\subsection{Points in the Plane, Experimentally}
\label{sec:5.2}
%%%%%%%%%%%%%%%%%%%%%%%%%%%%%%%%%%%%%%%%%%%%%%%%

This subsection presents experimental results for points in two dimensions.
As a substitute for $\Rspace^2$, we glue the sides of the unit square to form a torus and let $A \subseteq [0,1)^2$ be a stationary Poisson point process with intensity $\varrho > 0$.
Letting $\chi \colon A \to \{0,1\}$ be a random bi-coloring, we construct the chromatic Delaunay mosaic, $\Delaunay{}{\chi}$, while simulating general position of the points, if necessary, so the mosaic is simplicial.
Finally, we count the simplices of different types and write $N_{uv}$ for the number of simplices with $u$ vertices of color $0$ and $v$ vertices of color $1$.
For example, $N_{02}$ counts the edges in $\Delaunay{}{A_1}$, and $N_{11}$ counts the colorful edges in $\Delaunay{}{\chi}$.
Writing $m_p$ for the number of $p$-cells in the two mono-chromatic Voronoi tessellations, and $n_p$ for the number of colorful $p$-cells in $\Voronoi{}{\chi}$, we have
\begin{align}
  m_0  &=  N_{03} + N_{30} , ~~~~~~n_0  =  N_{13} + N_{22} + N_{31}, 
   \label{eqn:m0n0d2} \\
  m_1  &=  N_{02} + N_{20} , ~~~~~~n_1  =  N_{12} + N_{21}, 
    \label{eqn:m1n1d2} \\
  m_2  &=  N_{01} + N_{10} , ~~~~~~n_2  =  N_{11};
    \label{eqn:mpnp_d2_s1}
\end{align}
see Table~\ref{tbl:simplices_d2_s1} for some computed averages.
By symmetry, $N_{uv} = N_{vu}$ in expectation, so almost half the entries in this table are redundant.
\begin{table}[hbt]
    \centering
        \begin{tabular}{rrrr}
      \multicolumn{4}{c}{Average \#Simplices} \\
      \multicolumn{1}{c}{Vertices} & \multicolumn{1}{c}{Edges} & \multicolumn{1}{c}{Triangles} & \multicolumn{1}{c}{Tetrahedra} \\
      \hline
       $N_{01} = 499.7$ & $N_{02} = 1499.0$ & $N_{03} =  999.3$ & $N_{13} =  999.3$ \\
       $N_{10} = 500.3$ & $N_{11} = 2274.8$ & $N_{12} = 2773.8$ & $N_{22} = 1274.8$ \\
                         & $N_{20} = 1501.0$ & $N_{21} = 2775.8$ & $N_{31} =  1000.7$ \\
                         &                    & $N_{30} =  1000.7$ &    
    \end{tabular}
    \caption{The average number of simplices of each type---computed over $100$ repeats of the experiment---in the chromatic Delaunay mosaic of a randomly bi-colored stationary Poisson point process with intensity $\intensity = 1000$ in $[0,1)^2$.}
    \label{tbl:simplices_d2_s1}
\end{table}

\begin{table}[hbt]
    \centering
    \begin{tabular}{r||rrr|rr}
                  & \multicolumn{3}{c|}{\#Colorful Tetrahedra}
                  & \multicolumn{2}{c}{\#Crossings} \\
      \multicolumn{1}{c||}{$\intensity$} & \multicolumn{1}{c}{Min} & \multicolumn{1}{c}{Max} & \multicolumn{1}{c|}{Avg} & \multicolumn{1}{c}{Avg} & \multicolumn{1}{c}{StDev} \\
      \hline \hline
      1000  &  2999  &  3513  & 3265.8  & 1.2711  & 0.0466\\
      2000  &  6121  &  6872  & 6562.0  & 1.2774  & 0.0321\\
      5000  &  15812  &  17025  & 16393.9  & 1.2753  & 0.0201\\
      10000  &  31855  &  33468  & 32744.3  & 1.2731  & 0.0152
    \end{tabular}
    \caption{The minimum, maximum, average number of colorful Delaunay tetrahedra of a bi-chromatic stationary Poisson point process with intensities from $1000$ to $10000$ in $[0,1)^2$.
    \emph{Right:} the mean and standard deviation of the normalized crossing density, $(n_0 - m_0)/\intensity$.}
    \label{tbl:statistics_d2_s1}
\end{table}

\medskip
Table~\ref{tbl:statistics_d2_s1} shows more detailed statistics for the colorful tetrahedra, which correspond to the vertices in the overlay of the two mono-chromatic Voronoi tessellations.
To facilitate the comparison between the mono-chromatic and chromatic Delaunay mosaics, we also consider the surplus of vertices in the chromatic Voronoi tessellation, by which we mean $n_0 - m_0$.
Since $N_{03} = N_{13}$ and $N_{30} = N_{31}$, the surplus is the number of crossings, $n_0 - m_0 = N_{22}$, and by dividing with the intensity, we get an approximation of the \emph{normalized crossing density}, $(n_0 - m_0) / \varrho$.
In our experiments, the latter agrees with the constant in Theorem~\ref{thm:density_of_crossings_in_plane}, which for $s+1 = 2$ is $\frac{4}{\pi} = 1.27\ldots$.
While the standard deviation shrinks with increasing intensity, the approximation of the normalized crossing density does not seem to be affected by the number of points used in the experiment.

\medskip
Moving on to a random tri-coloring, $\chi \colon A \to \{0,1,2\}$, we write $N_{uvw}$ for the number of simplices in $\Delaunay{}{\chi}$ with $u,v,w$ vertices of color $0,1,2$, respectively.
The number of $p$-cells in the mono-chromatic and chromatic Voronoi tessellations thus satisfy
\begin{align}
  m_0  &=  N_{003} + N_{030} + N_{300} ,  ~~~~~~n_0  =  N_{113} + N_{131} + N_{311} + N_{122} + N_{212} + N_{221}, 
    \label{eqn:m0n0s3d2} \\
  m_1  &=  N_{002} + N_{020} + N_{200} ,  ~~~~~~n_1  =  N_{112} + N_{121} + N_{211} , 
    \label{eqn:m1n1s3d2} \\
  m_2  &=  N_{001} + N_{010} + N_{100} ,  ~~~~~~n_2  =  N_{111} ;
    \label{eqn:mpnp_d2_s2}
\end{align}
see Table~\ref{tbl:simplices_d2_s2} for some computed averages.
We omit any detailed statistics for number of colorful $4$-simplices and just mention that $n_0 - m_0 = N_{022} + N_{202} + N_{220}$ counts the crossings, and that $(n_0 - m_0) / \varrho$ agrees with the constant in Theorem~\ref{thm:density_of_crossings_in_plane}, which for $s+1 = 3$ is $\frac{8}{\pi} = 2.54\ldots$.
\begin{table}[hbt]
    \centering
    \begin{tabular}{rrrrr}
      \multicolumn{5}{c}{Average \#Simplices} \\
      \multicolumn{1}{c}{Vertices} & \multicolumn{1}{c}{Edges} & \multicolumn{1}{c}{Triangles} & \multicolumn{1}{c}{Tetrahedra} & \multicolumn{1}{c}{4-Simplices} \\
      \hline
       $N_{001} = 337.0$ & $N_{002} = 1010.8$ & $N_{003} = 673.9$ & $N_{013} = 673.9$ & $N_{113} = 673.9$ \\
                         & $N_{011} = 1526.3$ & $N_{012} = 1864.5$ & $N_{022} = 853.7$ & $N_{122} = 853.7$ \\
                         &                   & $N_{111} = 3557.3$ & $N_{112} = 2716.4$ & 
    \end{tabular}
    \caption{The average number of simplices of each type---computed over $100$ repeats of the experiment---in the chromatic Delaunay mosaic of a tri-colored stationary Poisson point process with intensity $\intensity = 1000$ in $[0,1)^2$.
    Numbers implied by symmetry are omitted \Skip{The counts didn't really got better without perturbation, I guess the software doesn't handle all the non-perturbed cases well. For example, for $N_{002}$ the count ranges from 9535 to 22572 for 100 iterations.}}
    \label{tbl:simplices_d2_s2}
\end{table}
%%%%%%%%%%%%%%%%%%%%%%%%%%%%%%%%%%%%%%%%%%%%%%%%
\subsection{Points in Space, Experimentally}
\label{sec:5.3}
%%%%%%%%%%%%%%%%%%%%%%%%%%%%%%%%%%%%%%%%%%%%%%%%

This subsection presents experimental results for points in three dimensions.
As a substitute for $\Rspace^3$, we glue the sides of the unit cube to form a $3$-dimensional torus and let $A \subseteq [0,1)^3$ be a stationary Poisson point process with intensity $\varrho > 0$.
We begin with a random bi-coloring, $\chi \colon A \to \{0,1\}$, for which the $p$-cells in the mono-chromatic and chromatic Voronoi tessellations satisfy
\begin{align}
  m_0  &=  N_{04} + N_{40} ,  ~~~~~~n_0  =  N_{14} + N_{23} + N_{32} + N_{41} , 
    \label{eqn:m0n0d3} \\
  m_1  &=  N_{03} + N_{30} ,  ~~~~~~n_1  =  N_{13} + N_{22} + N_{31} , 
    \label{eqn:m1n1d3} \\
  m_2  &=  N_{02} + N_{20} ,  ~~~~~~n_2  =  N_{12} + N_{21} , 
    \label{eqn:m2n2d3} \\
  m_3  &=  N_{01} + N_{10} ,  ~~~~~~n_3  =  N_{11} ;
    \label{eqn:mpnp_d3_s1}
\end{align}
see Table~\ref{tbl:simplices_d3_s1} for some computed averages.
The crossings for two colors in three dimensions are between Voronoi edges and Voronoi polygons, which are counted by $n_0 - m_0 = N_{23} + N_{32}$.
\begin{table}[hbt]
    \centering
    \begin{tabular}{rrrrr}
      \multicolumn{5}{c}{Average \#Simplices} \\
      \multicolumn{1}{c}{Vertices} & \multicolumn{1}{c}{Edges} & \multicolumn{1}{c}{Triangles} & \multicolumn{1}{c}{Tetrahedra} & \multicolumn{1}{c}{4-simplices} \\
      \hline
      $N_{01} = 503.6$ & $N_{02} = 3912.8$ & $N_{03} =  6818.5$ & $N_{04} =  3409.2$ & $N_{14} = 3409.2$ \\
                        & $N_{11} = 5269.8$ & $N_{12} = 12441.8$ & $N_{13} = 11082.8$ & $N_{23} = 4264.4$ \\
                        &                    &                     & $N_{22} = 12793.8$ &
    \end{tabular}
    \caption{The average number of simplices of each type in the chromatic Delaunay mosaic of a bi-colored stationary Poisson point process with intensity $\intensity = 1000$ in $[0,1)^3$.
    Numbers implied by symmetry are omitted.}
    \label{tbl:simplices_d3_s1}
\end{table}

\medskip
We continue with a random tri-coloring, $\chi \colon A \to \{ 0,1,2 \}$, for which the chromatic Delaunay mosaic is a complex in $\Rspace^5$.
The $p$-cells of the mono-chromatic and chromatic Voronoi tessellations satisfy
\begin{align}
  m_0  &=  N_{004} + N_{040} + N_{400} ,  ~~~~~~n_0  =  N_{114} + N_{141} + N_{411} + N_{123} + N_{132} \nonumber \\
  &~~~~~~~~~~~~~~~~~~~~~~~~~~~~~~~~~~~~~~~~~+ N_{213} + N_{312} + N_{231} + N_{321} + N_{222}, 
    \label{eqn:m0n0s3d3} \\
  m_1  &=  N_{003} + N_{030} + N_{300} ,  ~~~~~~n_1  =  N_{113} + N_{131} + N_{311} + N_{122} + N_{212} + N_{221} , 
    \label{eqn:m1n1s3d3} \\
  m_2  &=  N_{002} + N_{020} + N_{200} ,  ~~~~~~n_2  =  N_{112} + N_{121} + N_{211},
    \label{eqn:m2n2s3d3} \\
  m_3  &=  N_{001} + N_{010} + N_{100} ,  ~~~~~~n_3  =  N_{111} ;
    \label{eqn:mpnp_d3_s2}
\end{align}
see Table~\ref{tbl:simplices_d3_s2} for some computed averages.
The crossings are either between an edge of one color and a polygon of another color, or between three polygons, one of each color, which are counted by $n_0 - m_0 = N_{123} + N_{132} + N_{213} + N_{312} + N_{231} + N_{321} + N_{222}$.
We remark that $N_{222}$ is the only count for which the previous subsection does not offer an analytic expression for its expected value.
\begin{table}[hbt]
    \centering
    \begin{tabular}{rrrrrr}
      \multicolumn{6}{c}{Average \#Simplices} \\
      \multicolumn{1}{c}{Vertices} & \multicolumn{1}{c}{Edges} & \multicolumn{1}{c}{Triangles} & \multicolumn{1}{c}{Tetrahedra} & \multicolumn{1}{c}{4-simplices} & \multicolumn{1}{c}{5-simplices} \\
      \hline
      \footnotesize{$N_{001} = 332.7$} 
      & \!\!\footnotesize{$N_{002} = 2585.6$} 
      & \!\!\footnotesize{$N_{003} = 4505.7$} 
      & \!\!\footnotesize{$N_{004} = 2252.8$} 
      & \!\!\footnotesize{$N_{014} = 2252.8$} 
      & \!\!\footnotesize{$N_{114} = 2252.8$} \\
      & \!\!\footnotesize{$N_{011} = 3491.8$} 
      & \!\!\footnotesize{$N_{012} = 8237.4$} 
      & \!\!\footnotesize{$N_{013} = 7331.2$} 
      & \!\!\footnotesize{$N_{023} = 2825.6$} 
      & \!\!\footnotesize{$N_{123} = 2825.6$} \\
    & & \!\!\footnotesize{$N_{111} = 12678.0$} 
      & \!\!\footnotesize{$N_{022} = 8478.8$} 
      & \!\!\footnotesize{$N_{113} = 10150.7$} 
      & \!\!\footnotesize{$N_{222} = 3217.8$} \\
  & & & \!\!\footnotesize{$N_{112} = 17092.8$} 
      & \!\!\footnotesize{$N_{122} = 11696.5$} &
    \end{tabular}
    \caption{The average number of simplices of each type in the chromatic Delaunay mosaic of a tri-colored stationary Poisson point process with intensity $\intensity = 1000$ in $[0,1)^3$.
    Numbers implied by symmetry are omitted.}
    \label{tbl:simplices_d3_s2}
\end{table}

\Skip{
  \HE{[[Remove blue when not needed any more:]]}
  \HE{Table~\ref{tbl:statistics_d3_s2} gives the statistics of the $5$-simplices for a tri-colored Poisson point process in $\Rspace^3$.}
  \begin{table}[hbt]
    \centering
    \begin{tabular}{r||rrr|rr}
                  & \multicolumn{3}{c|}{\#Colorful 5-Simplices}
                  & \multicolumn{2}{c}{\#Crossings} \\
      \multicolumn{1}{c||}{$\intensity$} & \multicolumn{1}{c}{Min} & \multicolumn{1}{c}{Max} &
      \multicolumn{1}{c|}{Avg} & \multicolumn{1}{c}{Avg} & \multicolumn{1}{c}{StDev} \\
      \hline \hline
      100  &  2105  &  3562  & 2762.6  & 20.7148  & 1.9477\\
      200  &  4512  &  6382  & 5396.5  & 20.2091  & 1.3575\\
      500  &  11799  &  15100  & 13492.9  & 20.2144  & 0.9354\\
      1000  &  24955  &  29168  & 26892.9  & 20.1439  & 0.6292
    \end{tabular}
    \caption{\HE{The minimum, maximum, average number of colorful Delaunay $5$-simplices of a tri-chromatic stationary Poisson point process with intensities from $100$ to $1000$ in $[0,1)^3$.
    \emph{Right:} the mean and standard deviation of the normalized crossing density, $(n_0 - m_0)/\intensity$.}}
    \label{tbl:statistics_d3_s2}
  \end{table}
}

%%\clearpage
%%%%%%%%%%%%%%%%%%%%%%%%%%%%%%%%%%%%%%%%%%%%%%%%
%%%%%%%%%%%%%%%%%%%%%%%%%%%%%%%%%%%%%%%%%%%%%%%%
\section{Discussion}
\label{sec:6}
%%%%%%%%%%%%%%%%%%%%%%%%%%%%%%%%%%%%%%%%%%%%%%%%
%%%%%%%%%%%%%%%%%%%%%%%%%%%%%%%%%%%%%%%%%%%%%%%%

This paper introduces chromatic Delaunay mosaics to study the mingling of points of different color classes in Euclidean space.
Our main results are structural---proving relations useful in the topological analysis of mingling---and combinatorial---arguing that the size of the mosaic is sufficiently small to be attractive in applications.
There are two questions suggested by the work in this paper we mention.
\medskip \begin{itemize}
  \item Given a tri-colored stationary Poisson point process in $\Rspace^3$, what is the density of crossings between three $2$-cells---one each from the Voronoi tessellations of the three color classes?
  Indeed, $d = s+1 = 3$ is the first case for which the density of crossings is not yet known.
  What if $d \geq 3$ and $s+1 \geq 3$?
  \item We prove that sets with few spherical $k$-sets have small expected overlays of mono-chromatic Voronoi tessellations; see Lemma~\ref{lem:spherical_k-set_and_overlay}.
  Is the converse also true?
  More generally, how are sets with few spherical $k$-sets, sets with colorings whose mono-chromatic Voronoi tessellations have small overlays, and dense sets under some notion of density related?
  \item For a dense set in $\Rspace^2$, we strengthen the linear bound on the overlay size from expected to worst case, so it holds for every coloring of the set.
  Is there a reasonable notion of density in three and higher dimensions, such that the overlay of the mono-chromatic Voronoi tessellations has linear size for every coloring of a dense set?
\end{itemize} \medskip
There are also open-ended research directions suggested by the work reported in this paper.
For example: how tolerant are our results to faults in the data, such as the misclassification of (biological) cells?
How much of a difference does the change of the color of a small number of points make to the size and structure of the chromatic Delaunay mosaic?
What is the variance of the overlay size assuming the coloring is random?

%%%%%%%%%%%%%%%%%%%%%%%%%%%%%%%%%%%%%%%%%%%%%%%%
\subsubsection*{Acknowledgements}
%%%%%%%%%%%%%%%%%%%%%%%%%%%%%%%%%%%%%%%%%%%%%%%%
{\footnotesize
  The fourth author thanks Boris Aronov for insightful discussions on the size of the overlay of Voronoi tessellations.}

%%\newpage
%%%%%%%%%%%%%%%%%%%%%%%%%%%

\end{document}

%% file: Figs/unions.pdf_t
\begin{picture}(0,0)%
\includegraphics{Figs/unions.pdf}%
\end{picture}%
\setlength{\unitlength}{4144sp}%
\begingroup\makeatletter\ifx\SetFigFont\undefined%
\gdef\SetFigFont#1#2#3#4#5{%
  \reset@font\fontsize{#1}{#2pt}%
  \fontfamily{#3}\fontseries{#4}\fontshape{#5}%
  \selectfont}%
\fi\endgroup%
\begin{picture}(12040,3692)(607,-3257)
\end{picture}%